\documentclass[11pt,a4paper]{article}
\usepackage[latin9]{inputenc}
\usepackage{amsmath}
\usepackage{amssymb}

\makeatletter

\title{\bf Vorticity, Helicity, Intrinsinc geometry for Navier-Stokes equations}
\author{Shizan Fang$^1$\footnote{Email: Shizan.Fang@u-bourgogne.fr.}
\quad Zhongmin Qian $^2$\footnote{Email: zhongmin.qian@maths.ox.ac.uk}
\vspace{3mm}\\
{\footnotesize $^1$I.M.B, Universit\'e de Bourgogne, BP 47870, 21078 Dijon, France}\\
{\footnotesize $^2$Mathematical Institute, University of Oxford, Oxford OX2 6GG, UK}\\
}
\usepackage{amsfonts}
\usepackage{amsthm}
\usepackage{array}
\usepackage{bm}
\usepackage{color}

\def\R{\mathbb{R}}
\def\RR{\mathcal{R}}
\def\E{\mathbb{E}}

\def\PP{\mathbf{P}}

\def\X{\mathcal{X}}

\def\L{\mathcal L}

\def\TT{\mathbf{T}}

\def\div{\textup{div}}
\def\Scal{\rm Scal}

\def\Ric{{\rm Ric}}
\def\ric{{\rm ric}}
\def\eps{\varepsilon}
\def\<{\langle}
\def\>{\rangle}
\def\Def{\rm Def\,}
\let \dis=\displaystyle
\let\ra=\rightarrow
\let\intt\triangleleft

\newtheorem{theorem}{Theorem}[section]
\newtheorem{lemma}[theorem]{Lemma}\newtheorem{corollary}[theorem]{Corollary}\newtheorem{proposition}[theorem]{Proposition}\newtheorem{remark}[theorem]{Remark}\newtheorem{definition}[theorem]{Definition}

\makeatother

\begin{document}
\maketitle \makeatletter 
\global\long\def\theequation{\thesection.\arabic{equation}}
 \@addtoreset{equation}{section} \makeatother 

\vspace{-4mm}
 
\begin{abstract}
We will consider the Navier-Stokes equation on a Riemannian manifold
$M$ with Ricci tensor bounded below, the involved Laplacian operator
is De Rham-Hodge Laplacian. The novelty of this work is to introduce
a family of connections which are related to solutions of the Navier-Stokes
equation, so that vorticity and helicity can be linked through the
associated time-dependent Ricci tensor in intrinsic way in the case
where $\dim(M)=3$. 
\end{abstract}
\vskip 4mm



\vskip 13mm

\section{Introduction}

\label{sect1}

The Navier-Stokes equation in a domain of $\mathbb{R}^{n}$ is a system
of partial differential equations
\begin{equation}
\partial_{t}u_{t}+(u_{t}\cdot\nabla)u_{t}-\nu\Delta u_{t}+\nabla p_{t}=0,\quad\nabla\cdot u_{t}=0,\quad u|_{t=0}=u_{0},\label{eq0.1}
\end{equation}
which describes the evolution of the velocity $u_{t}$ and the pressure
$p_{t}$ of an incompressible viscous fluid with kinematic viscosity
$\nu>0$. The model of periodic boundary conditions for \eqref{eq0.1}
on a torus $\mathbb{T}^{n}$ has been introduced to simplify mathematical
considerations. In \cite{EM}, Navier-Stokes equations on a compact
Riemannian manifold $M$ have been considered using the framework
of the group of diffeomorphisms of $M$ initiated by V. Arnold in
\cite{Arnold}; where the Laplace operator involved in the text of
\cite{EM} is de Rham-Hodge Laplacian $\square$, however, the authors
said in the note added in proof that the convenient Laplace operator
comes from deformation tensor. 

In this article, we would like to explore the rich geometry coded in
the Navier-Stokes equation on a manifold.  

Let $\nabla$ be the Levi-Civita connection on $M$. For a vector
field $A$ on $M$, the deformation tensor $\Def(A)$ is a symmetric
tensor of type $(0,2)$ defined by

\begin{equation}
(\Def A)(X,Y)=\frac{1}{2}\Bigl(\langle\nabla_{X}A,Y\rangle+\langle\nabla_{Y}A,X\rangle\Bigr),\quad X,Y\in{\mathcal{X}}(M),\label{eq0.2}
\end{equation}
where ${\mathcal{X}}(M)$ is the space of vector fields on $M$. Then
$\Def:TM\ra S^{2}T^{*}M$ maps a vector field to a symmetric tensor
of type $(0,2)$. Let $\Def^{*}:S^{2}T^{*}M\ra TM$ be the adjoint
operator. In \cite{MT} or in \cite{Taylor} (see page 493), the authors
considered the following Laplacian

\begin{equation}
\hat{\square}=2\Def^{*}\,\Def.\label{eq0.3}
\end{equation}

They considered the Navier-Stokes equation with viscosity described
by $\hat{\square}$, namely
\begin{equation}
\partial_{t}u_{t}+\nabla_{u_{t}}u_{t}+\nu\hat{\square}u_{t}=-\nabla p_{t},\quad\div(u_{t})=0,\quad u|_{t=0}=u_{0},\label{eq0.4}
\end{equation}
The reader may also refer to \cite{Pier} in which the author considered
the same equation as \eqref{eq0.4} on a complete Riemnnian manifold
with negative curvature.  Variational principles in the class of
incompressible Brownian martingales in the spirit of \cite{Arnold}
were established recently in \cite{Cruzeiro, AC1,AC2,ACF} for the Navier-Stokes
equation \eqref{eq0.4}.

\vskip 3mm In this work, we will concerned with a complete Riemannian
manifold $M$ of dimension $n$, with Ricci curvature bounded from
below. We are interested in the following Navier-Stokes equation on
$M$ defined with the De Rham-Hodge Laplacian $\square$,

\begin{equation}
\begin{cases}
\partial_{t}u_{t}+\nabla_{u_{t}}u_{t}+\nu\square u_{t}=-\nabla p_{t},\\
\div(u_{t})=0,\quad u|_{t=0}=u_{0},
\end{cases}\label{eq1.1}
\end{equation}
where $u(x,t)$ denotes the velocity vector field at time $t$, and
$p(x,t)$ models the pressure. If no confusion may arise, we will
use $u_{t}$ (resp. $p_{t}$) to denote the vector field $u(\cdot,t)$
(resp. $p(\cdot,t)$) for each $t$.

There are a few works \cite{Kobayashi,TemamW} which support this
choice of $\square$. The probabilistic representation formulae behave
better with Navier-Stokes equation \eqref{eq1.1} (see \cite{ConstantinI,FangLuo,Fang}).
Our preference here for $\square$ is motivated by its good geometric
behavior and its deep links with Stochastic Analysis. See for example
\cite{Bakry,Bismut1,Bismut2,Driver,DriverTh,Elworthy,ElworthyLi,ELL,FangMalliavin,IW,Kunita,Malliavin,Stroock}.
From the view of kinetic mechanics, the viscosity effect of a non-homogeneous
fluid should be mathematically described by the Bochner Laplacian
of the velocity vector field, where the metric tensor describes the
local viscosity distribution. On the other hand, the de Rham-Hodge
Laplacian operating on one forms is mathematically more appealing.
By invoking de Rham-Hodge Laplacian in the model, according to the
Bochner identity, one then produces a no-physical additional term
which is however linear in the velocity. An additional linear term
in the Navier-Stokes equation will not alter the fundamental difficulty,
nor to alter the physics of the fluid flows, which justify the use
of de Rham-Hodge Laplacian. There is also a good reason too to consider
Navier-Stokes equations on manifolds, if one wants to model the global
behavior of the pacific ocean climate for example. 

\vskip 3mm Let's first say a few words on the definition of $\square$
on vector fields. There is a one-to-one correspondence between the
space of vector fields $\X(M)$ and that of differential 1-forms $\Lambda^{1}(M)$.
Given a vector field $A$ (resp. differential 1-form $\omega$), we
shall denote by $\tilde{A}$ (resp. $\omega^{\sharp}$) the corresponding
differential 1-form (resp. vector field). To see more intuitively
these correspondences, let's explain on a local chart $U$: as usual,
we denote by $\{\frac{\partial}{\partial x_{1}},\ldots,\frac{\partial}{\partial x_{n}}\}$
the basis of the tangent space $T_{x}M$ and by $\{dx^{1},\ldots,dx^{n}\}$
the dual basis of $T_{x}^{*}M$, called the co-tangent space at $x$,
that is, $\dis dx^{i}(\frac{\partial}{\partial x_{j}})=\delta_{ij}$.
The inner product in $T_{x}M$ as well as the one in the dual space
$T_{x}^{*}M$ will be denoted by $\langle\ ,\ \rangle$, while the
duality between $T_{x}^{*}M$ and $T_{x}M$ will be denoted by $(\ ,\ )$.
Set $g_{ij}=\langle\frac{\partial}{\partial x_{i}},\frac{\partial}{\partial x_{j}}\rangle$.
Let $u$ be a vector field on $M$, on $U$, $\dis u=\sum_{i=1}^{n}u_{i}\frac{\partial}{\partial x_{i}}$,
then $\tilde{u}$ admits the expression 
\[
\tilde{u}=\sum_{i=1}^{n}\Bigl(\sum_{j=1}^{n}g_{ij}u_{j}\Bigr)dx^{i}.
\]

Let $g^{ij}=\langle dx^{i},dx^{j}\rangle$. Then the matrix $(g^{ij})$
is the inverse matrix of $(g_{ij})$. For a differential $1$-form
$\dis\omega=\sum_{j=1}^{n}\omega_{j}dx^{j}$, the associated vector
field $\omega^{\#}$ has the expression 
\[
\omega^{\#}=\sum_{i=1}^{n}\Bigl(\sum_{\ell=1}^{n}g^{i\ell}\omega_{\ell}\Bigr)\,\frac{\partial}{\partial x_{i}}.
\]

Concisely

\[
(\omega,A)=\langle\omega^{\#},A\rangle=\langle\omega,\tilde{A}\rangle,\quad A\in\X(M),\ \omega\in\Lambda^{1}(M).
\]

\vskip 3mm Now for $A\in\X(M)$, the De-Rham Hodge Laplacian $\square A$
is defined by 
\[
\square A=(\square\tilde{A})^{\#},\quad\square=dd^{\ast}+d^{\ast}d,
\]
where $d^{\ast}$ is adjoint operator of exterior derivative $d$.
Then we have the following relation 
\[
\int_{M}(\square\omega,A)\,dx=\int_{M}\langle\square\omega,\tilde{A}\rangle\,dx=\int_{M}\langle\omega,\square\tilde{A}\rangle\,dx=\int_{M}(\omega,\square A)\,dx
\]
where $dx$ denotes the Riemannian measure on $M$. The classical
Bochner-Weitzenböck reads as

\begin{equation}
\square A=-\Delta A+\Ric(A),\quad A\in\X(M),\label{Weitzenbock1}
\end{equation}
where $\Ric$ is the Ricci tensor on $M$ and $\Delta A={\rm Trace}(\nabla\nabla A)$,
characterized by

\begin{equation}
-\int_{M}\langle\Delta A,A\rangle\,dx=\int_{M}|\nabla A|^{2}\,dx.\label{Weitzenbock2}
\end{equation}

Let $T:\X(M)\ra\X(M)$ be a tensor of type $(1,1)$, and denote by
$T^{\#}:\Lambda^{1}(M)\ra\Lambda^{1}(M)$ its adjoint defined by

\begin{equation}
(T^{\#}\omega,A)=(\omega,T(A)),\quad A\in\X(M),\label{eq0.8}
\end{equation}
where we used notation $\Lambda^{p}(M)$ to denote the space of differential
$p$-forms on $M$.

\vskip 3mm

In the space of $\R^{3}$, the inner product between two vectors $u,v$
will be noted by $u\cdot v$. The vorticity $\xi_{t}$ of a velocity
$u_{t}$ is a vector field defined as $\xi_{t}=\nabla\times u_{t}$.
When $u_{t}$ is a solution to Navier-Stokes equation \eqref{eq0.1},
then $\xi_{t}$ satisfies the following heat equation

\begin{equation}
\frac{d\xi_{t}}{dt}+\nabla_{u_{t}}\xi_{t}-\nu\Delta\xi_{t}=\nabla_{\xi_{t}}^{s}u_{t}\label{eq0.9}
\end{equation}
where $\nabla^{s}u_{t}$ is the symmetric part of $\nabla u_{t}$,
such that $\nabla_{\xi_{t}}^{s}u_{t}\cdot v=\Def u_{t}(\xi_{t},v)$
with $\Def$ introduced in \eqref{eq0.2}. How to interpret the term
$\nabla_{\xi_{t}}^{s}u_{t}$ ? From \eqref{eq0.9}, a formal computation
leads to

\begin{equation}
\frac{1}{2}\frac{d}{dt}\int_{\R^{3}}|\xi_{t}|^{2}\,dx+\nu\int_{\R^{3}}|\nabla\xi_{t}|^{2}\,dx=\int_{\R^{3}}\Def(u_{t})(\xi_{t},\xi_{t})\,dx.\label{eq0.10}
\end{equation}

\vskip 3mm Since K. Itô introduced the tool of stochastic parallel
translations along paths of Brownian motion on a Riemannian manifold,
especially after the works by Eells, Elworthy, Malliavin and Bismut
(see for example \cite{Malliavin,Elworthy2,Bismut2}), there are profound
involvements of Stochastic Analysis in the study of linear second
order partial differential equations and in Riemannian geometry \cite{Bakry,Stroock,IW,Li}.
The purpose of this work is to geometrically explain the right hand
side of \eqref{eq0.10}. To this end, we will consider Navier-Stokes
equation in a geometric framework in order that suitable geometric
meaning could be found. 

\vskip 3mm In what follows, we present the organisation of the paper
and main results. In Section \ref{sect2}, first we follow more or
less the exposition of \cite{Taylor}. To a solution $u_{t}$ to Navier-Sokes
equaion \eqref{eq1.1}, we associate a differential $2$-form $\tilde{\omega}_{t}$
which is the exterior derivative of $\tilde{u}_{t}$; a heat equation
for $\tilde{\omega}_{t}$ will be established with involvement of
$\nabla^{s}u_{t}$. When $M$ is of dimension $3$, the Hodge star
$\ast$ operator sends $\tilde{\omega}_{t}$ to a differential $1$-form
$\omega_{t}$. In flat case of $\R^{3}$, $\omega_{t}=\widetilde{\nabla\times u_{t}}$.
We call such $\omega_{t}$ the vorticity of $u_{t}$; a heat equation
for $\omega_{t}$ is also obtained in Section \ref{sect2}. In second
part of Section \ref{sect2}, the a priori evolution equation for
$\omega_{t}$ is established. Using heat semi-group $e^{-t\square}$
on differential forms as well as Bismut formulae, the existence of
weak solutions in the sense of Leray to Navier-Stokes equation \eqref{eq1.1}
over any intervall $[0,T]$ is proved under suitable hypothesis on
boundedness of Ricci tensor : to our knowledge, these results are
new while comparing to recent results obtained in \cite{Pier}. In
Section \ref{sect3}, we give an exposition of the involvement of
Stochastic Analysis on Riemannian manifolds; stochastic differential
equations on $M$, defining the Brownian motion with drift $u\in L^{2}([0,T],H^{1}(M))$
of divergence free is proved to be stochastic complete; then $\omega_{t}$
admits a probabilistic representation. By introducing a suitable metric
compatible affine connection on $M$, a Brownian motion with drift
$u$ on $M$ can be obtained by rolling without friction flat Brownian
motion of $\R^{n}$ on $M$ with respect to it : it was a main idea
in \cite{Malliavin,Elworthy2}, and well developed in \cite{IW}.
So to a velocity $u_{t}$, we associate a metric compatible connection
$\nabla^{t}$ on $M$, which admits the following global expression

\[
\nabla_{X}^{t}Y=\nabla_{X}Y-\frac{2}{n-1}K_{t}(X,Y),\quad X,Y\in\X(M)
\]
where $K_{t}(X,Y)=\langle Y,u_{t}\rangle X-\langle X,Y\rangle u_{t}$:
it gives rise to a connection with torsion $T^{t}$ which is not of
skew-symmetric. Section \ref{sect4} is devoted to compute the associated
intrinsic Ricci tensor $\hat{\Ric}^{t}$ which was first introduced
by B. Driver in \cite{Driver} as follows:

\[
\widehat{\Ric}^{t}(X)=\Ric^{t}(X)+\sum_{i=1}^{n}(\nabla_{e_{i}}^{t}T^{t})(X,e_{i}),
\]
where $\Ric^{t}$ is the Ricci tensor associated to $\nabla^{t}$
and $\{e_{1},\ldots,e_{n}\}$ is an orthonormal basis at tangent spaces.
The formula \eqref{eq0.10} has the following geometric counterpart
for $3D$ Riemannian manifold $M$,

\begin{equation}
\frac{1}{2}\frac{d}{dt}\int_{M}|\omega_{t}|^{2}\,dx+\nu\int_{M}|\nabla\omega_{t}|^{2}\,dx=\frac{1}{2\nu}\int_{M}(\omega_{t},u_{t})^{2}\,dx-\nu\int_{M}(\widehat{\Ric}^{t,\#}\omega_{t},\omega_{t})\ dx.\label{eq0.11}
\end{equation}

As well as vorticity $\omega_{t}$ is not orthogonal to velocity $u_{t}$,
a phenomenon of helicity $(\omega_{t},u_{t})$ will appear. Formula
\eqref{eq0.11} says how helicity and intrinsic Ricci tensor fit into
the evolution of vorticity in time and in space. Section \ref{sect5}
is devoted to interpretation of main results obtained in Section \ref{sect4}
in the framework of vector calculus. Finally in Section \ref{sect6},
we collect and prove technical results used previously.

\section{Vorticity, Helicity and their evolution equations}

\label{sect2}

Let $u_{t}$ be a (smooth) solution to the Navier-Stokes equation
on $M$, 
\begin{equation}
\partial_{t}u_{t}+\nabla_{u_{t}}u_{t}+\nu\square u_{t}=-\nabla p_{t},\quad\div(u_{t})=0,\ u|_{t=0}=u_{0}.\label{eq2.1}
\end{equation}

Transforming Equation \eqref{eq2.1} into differential forms, $\tilde{u}_{t}$
satisfies

\begin{equation}
\begin{cases}
\partial_{t}\tilde{u}_{t}+\nabla_{u_{t}}\tilde{u}_{t}+\nu\square\tilde{u}_{t}=-dp_{t},\\
d^{\ast}\tilde{u}_{t}=0,\quad\tilde{u}|_{t=0}=\tilde{u}_{0}.
\end{cases}\label{eq2.2}
\end{equation}

Let 
\begin{equation}
\tilde{\omega}_{t}=d\tilde{u}_{t},\label{eq2.2-1}
\end{equation}
which is a differential $2$-form. For vector fields $X,v$ on $M$,
Lie derivative $\L$ satisfies the product rule, that is,

\[
\L_{v}(\tilde{u},X)=(\L_{v}\tilde{u},\ X)+(\tilde{u},\L_{v}X),
\]
where
\[
\L_{v}(\tilde{u},X)=(\nabla_{v}\tilde{u},\ X)+(\tilde{u},\nabla_{v}X).
\]
By taking $v=u$, we get 
\[
(\L_{u}\tilde{u}-\nabla_{u}\tilde{u},\ X)=(\tilde{u},\nabla_{u}X-\L_{u}X)=(\tilde{u},\nabla_{X}u)=\langle u,\nabla_{X}u\rangle=\frac{1}{2}(d|u|^{2},\ X)
\]
which yields that 
\begin{equation}
\L_{u}\tilde{u}-\nabla_{u}\tilde{u}=\frac{1}{2}d|u|^{2}.\label{eq2.3}
\end{equation}
By definition $\L_{u}=i_{u}d+di_{u}$ where $i_{u}$ denotes the interior
product by $u$, so the exterior derivative $d$ commutes with $\L_{u}$
since $d\L_{u}=\L_{u}d=di_{u}d$, and therefore by using \eqref{eq2.3},
\[
d\nabla_{u}\tilde{u}=d\L_{u}\tilde{u}=\L_{u}d\tilde{u}.
\]
It is obvious that $\square d=d\square$. Then by acting $d$ on the
two sides of \eqref{eq2.2}, we get

\begin{equation}
\begin{cases}
\partial_{t}\tilde{\omega}_{t}+\L_{u_{t}}\tilde{\omega}_{t}+\nu\square\tilde{\omega}_{t}=0,\\
\tilde{\omega}|_{t=0}=\tilde{\omega}_{0}.
\end{cases}\label{eq2.4}
\end{equation}

\begin{remark}\label{remark2.1} Since $d^{\ast}\tilde{u}=0$, by
definition \eqref{eq2.2-1}, $d^{\ast}\tilde{\omega}=d^{\ast}d\tilde{u}=\square\tilde{u}$,
and therefore, as $\square$ admits a spectral gap, $\tilde{u}$ can
be solved by 
\[
\tilde{u}=\square^{-1}(d^{\ast}\tilde{\omega}).
\]
\end{remark}

\vskip 3mm It is sometimes more convenient to use covariant derivatives.
To do this, let $\beta$ be a differential $p$-form and $T:\X(M)\ra\X(M)$
be a tensor of type $(1,1)$. Define for $X_{1},\ldots,X_{p}$, 
\begin{equation}
(\beta\intt T)(X_{1},\ldots,X_{p})=\beta(T(X_{1}),X_{2},\ldots,X_{p})+\ldots+\beta(X_{1},\ldots,X_{p-1},T(X_{p})).\label{eq2.4-1}
\end{equation}
If $\beta$ is a $2$-form and $T=\nabla u$, then for $X,Y\in\X(M)$,
\begin{equation}
(\beta\intt\nabla u)(X,Y)=\beta(\nabla_{X}u,Y)+\beta(X,\nabla_{Y}u).\label{eq2.5}
\end{equation}

In the same way as for proving \eqref{eq2.3}, we have 
\[
(\L_{v}\beta-\nabla_{v}\beta)(X,Y)=\beta(\nabla_{X}v,Y)+\beta(X,\nabla_{Y}v)=(\beta\intt\nabla v)(X,Y).
\]

Now replacing $\L_{u}\tilde{\omega}$ by $\nabla_{u}\tilde{\omega}+\tilde{\omega}\intt\nabla u$
in \eqref{eq2.4}, we obtain the following form 
\begin{equation}
\begin{cases}
\partial_{t}\tilde{\omega}_{t}+\nabla_{u_{t}}\tilde{\omega}_{t}+\nu\square\tilde{\omega}_{t}=-\tilde{\omega}_{t}\intt\nabla u_{t},\\
\tilde{\omega}|_{t=0}=\tilde{\omega}_{0}.
\end{cases}\label{eq2.6}
\end{equation}

\begin{proposition}\label{prop2.1} Let $\nabla^{sk}u$ be the skew-symmetric
part of $\nabla u$, that is, 
\[
\langle\nabla^{sk}u,X\otimes Y\rangle=\frac{1}{2}\bigl(\langle\nabla_{X}u,Y\rangle-\langle\nabla_{Y}u,X\rangle\bigr).
\]
Then $\dis\tilde{\omega}\intt\nabla^{sk}u=0$. \end{proposition}

\begin{proof} Fix $x\in M$ and let $\{e_{1},\ldots,e_{n}\}$ be
an orthonormal basis of $T_{x}M$. Then 
\[
\begin{split}\nabla_{X}^{sk}u & =\sum_{i,j=1}^{n}\langle\nabla_{e_{i}}^{sk}u,e_{j}\rangle\langle X,e_{i}\rangle\,e_{j}\\
 & =\sum_{i,j=1}^{n}d\tilde{u}(e_{i},e_{j})\langle X,e_{i}\rangle\,e_{j}=\sum_{j=1}^{n}\tilde{\omega}(X,e_{j})\,e_{j},
\end{split}
\]
so that 
\[
\tilde{\omega}(\nabla_{X}^{sk}u,Y)=\sum_{j=1}^{n}\tilde{\omega}(X,e_{j})\tilde{\omega}(e_{j},Y)=\tilde{\omega}(\nabla_{Y}^{sk}u,X).
\]
Combing these relations and Definition \eqref{eq2.4-1}, we have 
\[
(\tilde{\omega}\intt\nabla^{sk}u)(X,Y)=\tilde{\omega}(\nabla_{X}^{sk}v,Y)+\tilde{\omega}(X,\nabla_{Y}^{sk}v)=0.
\]
\end{proof}

Let $\nabla^{s}u$ be the symmetric part of $\nabla u$, that is 
\[
\langle\nabla^{s}u,X\otimes Y\rangle=\frac{1}{2}\bigl(\langle\nabla_{X}u,Y\rangle+\langle\nabla_{Y}u,X\rangle\bigr).
\]
$\nabla^{s}u$ is called the rate of strain tensor in the literature
on fluid dynamics. Therefore Equation \eqref{eq2.6} can be written
in the following form:

\begin{equation}
\begin{cases}
\partial_{t}\tilde{\omega}_{t}+\nabla_{u_{t}}\tilde{\omega}_{t}+\nu\square\tilde{\omega}_{t}=-\tilde{\omega}_{t}\intt\nabla^{s}u_{t},\\
\tilde{\omega}|_{t=0}=\tilde{\omega}_{0}.
\end{cases}\label{eq2.7}
\end{equation}

\vskip 3mm In the case where $\dim(M)=2$ or $3$, Equation \eqref{eq2.7}
can be simplified using Hodge star operator $*$. Assume that $M$
is oriented and $\omega_{n}$ is the $n$-form of Riemannian volume,
let $\omega=*\tilde{\omega}$, which is a $(n-2)$ form such that

\[
\tilde{\omega}\wedge\alpha=\langle\omega,\alpha\rangle_{\Lambda^{n-2}}\ \omega_{n},\quad\hbox{\rm forany}\alpha\in\Lambda^{n-2}(M),
\]
or

\[
\beta\wedge\ *\tilde{\omega}=\langle\tilde{\omega},\beta\rangle_{\Lambda^{2}}\ \omega_{n},\quad\hbox{\rm forany}\beta\in\Lambda^{2}(M).
\]

\vskip 3mm

\begin{proposition}\label{prop2.3} Let $\omega$ be a $p$-form
on $M$ and $\div(u)=0$. Then $\nabla_{u}(*\omega)=*(\nabla_{u}\omega)$.
\end{proposition}

\begin{proof} Let $\beta$ be a $p$-form. Then $\beta\wedge*\omega=\langle\beta,\omega\rangle\ \omega_{n}$.
Taking the covariant derivative with respect to $u$, the left hand
side gives 
\[
\nabla_{u}\beta\wedge(*\omega)+\beta\wedge\nabla_{u}(*\omega)=\langle\nabla_{u}\beta,\omega\rangle\ \omega_{n}+\beta\wedge\nabla_{u}(*\omega),
\]
while the right hand side gives 
\[
\langle\nabla_{u}\beta,\omega\rangle\ \omega_{n}+\langle\beta,\nabla_{u}\omega\rangle\ \omega_{n}=\langle\nabla_{u}\beta,\omega\rangle\ \omega_{n}+\beta\wedge*\nabla_{u}\omega
\]
as $\nabla_{u}\omega_{n}=0$. Therefore $\beta\wedge\nabla_{u}(*\omega)=\beta\wedge(*\nabla_{u}\omega)$
holds for any $p$-form $\beta$, the result follows. \end{proof}

\begin{proposition}\label{prop2.4} Assume $\dim(M)=3$. Then 
\begin{equation}
*\Bigl(\tilde{\omega}_{t}\intt\nabla^{s}u\Bigr)=-(*\tilde{\omega}_{t})\intt\nabla^{s}u.\label{eq2.8}
\end{equation}
\end{proposition}

\begin{proof} Fix $x\in M$; let $\{e_{1},e_{2},e_{3}\}$ be an orthonormal
basis of $T_{x}M$, $\{\tilde{e}_{1},\tilde{e}_{2},\tilde{e}_{3}\}$
be the dual basis of $T_{x}^{*}M$. Let $\{i_{1},i_{2},i_{3}\}$ be
a direct permutation of $\{1,2,3\}$, and $\omega=\tilde{e}_{i_{1}}\wedge\tilde{e}_{i_{2}}$.
Then

\[
\begin{split}(\omega\intt\nabla^{s}u)(X,Y) & =\Bigl((\nabla_{X}^{s}u)_{i_{1}}Y_{i_{2}}-Y_{i_{1}}\,(\nabla_{X}^{s}u)_{i_{2}}\Bigr)+\Bigl((\nabla_{Y}^{s}u)_{i_{2}}X_{i_{1}}-X_{i_{2}}\,(\nabla_{Y}^{s}u)_{i_{1}}\Bigr)\\
 & =\sum_{j=1}^{3}\Bigl[(\nabla_{e_{j}}^{s}u)_{i_{1}}X_{j}Y_{i_{2}}-(\nabla_{e_{j}}^{s}u)_{i_{2}}X_{j}Y_{i_{1}}+(\nabla_{e_{j}}^{s}u)_{i_{2}}X_{i_{1}}Y_{j}-(\nabla_{e_{j}}^{s}u)_{i_{1}}X_{i_{2}}Y_{j}\Bigr]\\
 & =\sum_{j=1}^{3}(\nabla_{e_{j}}^{s}u)_{i_{1}}\Bigl(X_{j}Y_{i_{2}}-X_{i_{2}}Y_{j}\Bigr)+\sum_{j=1}^{3}(\nabla_{e_{j}}^{s}u)_{i_{2}}\Bigl(X_{i_{1}}Y_{j}-X_{j}Y_{i_{1}}\Bigr).
\end{split}
\]
It follows that 
\[
\omega\intt\nabla^{s}u=\sum_{j=1}^{3}(\nabla_{e_{j}}^{s}u)_{i_{1}}\tilde{e}_{j}\wedge\tilde{e}_{i_{2}}+\sum_{j=1}^{3}(\nabla_{e_{j}}^{s}u)_{i_{2}}\tilde{e}_{i_{1}}\wedge\tilde{e}_{j}.
\]
More precisely 
\[
\begin{split}\omega\intt\nabla^{s}u & =(\nabla_{e_{i_{1}}}^{s}u)_{i_{1}}\tilde{e}_{i_{1}}\wedge\tilde{e}_{i_{2}}+(\nabla_{e_{i_{2}}}^{s}u)_{i_{2}}\tilde{e}_{i_{1}}\wedge\tilde{e}_{i_{2}}\\
 & +(\nabla_{e_{i_{3}}}^{s}u)_{i_{1}}\tilde{e}_{i_{3}}\wedge\tilde{e}_{i_{2}}+(\nabla_{e_{i_{3}}}^{s}u)_{i_{2}}\tilde{e}_{i_{1}}\wedge\tilde{e}_{i_{3}}.
\end{split}
\]

Since $\dis\sum_{j=1}^{3}(\nabla_{e_{i_{j}}}^{s}u)_{i_{j}}=\hbox{\rm Trace}(\nabla u)=\div(u)=0$,
therefore finally we get 
\begin{equation}
*(\omega\intt\nabla^{s}u)=-(\nabla_{e_{i_{3}}}^{s}u)_{i_{1}}\tilde{e}_{i_{1}}-(\nabla_{e_{i_{3}}}^{s}u)_{i_{2}}\tilde{e}_{i_{2}}-(\nabla_{e_{i_{3}}}^{s}u)_{i_{3}}\tilde{e}_{i_{3}}.\label{eq2.9}
\end{equation}

On the other hand, $*\omega=\tilde{e}_{i_{3}}$, so that 
\[
(*\omega)\intt(\nabla^{s}u)(X)=(*\omega)(\nabla_{X}^{s}u)=\sum_{j=1}^{3}(\nabla_{e_{j}}^{s}u)_{i_{3}}X_{j}.
\]
It follows that 
\begin{equation}
(*\omega)\intt(\nabla^{s}u)=(\nabla_{e_{i_{1}}}^{s}u)_{i_{3}}\tilde{e}_{i_{1}}+(\nabla_{e_{i_{2}}}^{s}u)_{i_{3}}\tilde{e}_{i_{2}}+(\nabla_{e_{i_{3}}}^{s}u)_{i_{3}}\tilde{e}_{i_{3}}\label{eq2.10}
\end{equation}
Now combing \eqref{eq2.9}, \eqref{eq2.10}, and by symmetry of $\nabla^{s}u$,
we get \eqref{eq2.8}. \end{proof}

\begin{corollary} Let $\dim(M)=3$ and $\omega_{t}=*\tilde{\omega}_{t}$.
Then 
\begin{equation}
\partial_{t}\omega_{t}+\nabla_{u}\omega_{t}+\nu\square\omega_{t}=\omega_{t}\intt(\nabla^{s}u_{t}).\label{eq2.11}
\end{equation}

\end{corollary}

\begin{proof} First note that $\square*=*\square$ (see \cite{Warner},
p. 221), so \eqref{eq2.11} follows from Proposition \ref{prop2.3}
and Proposition \ref{prop2.4}. \end{proof}

\begin{remark} Since $**=(-1)^{p(n-p)}$ on $p$-form, so for $n=3$,
$\tilde{\omega}_{t}=*\omega_{t}$ and in the case where $\square$
admits a spectral gap, the following relation holds 
\begin{equation}
\tilde{u}_{t}=\square^{-1}\bigl(d^{\ast}(*\omega_{t})\bigr).\label{eq2.13}
\end{equation}
\end{remark}

\begin{proposition}\label{prop2.7}In the smooth case, it holds 
\begin{equation}
\frac{1}{2}\frac{d}{dt}\int_{M}|u_{t}|^{2}\,dx+\nu\int_{M}|\nabla u_{t}|^{2}\,dx=-\nu\int_{M}\langle\Ric\,u_{t},u_{t}\rangle\,dx.\label{eq2.15}
\end{equation}
\end{proposition}

\begin{proof} Remark first that $\dis\int_{M}\langle\nabla_{u_{t}}u_{t},u_{t}\rangle\,dx=\frac{1}{2}\int_{M}\L_{u_{t}}|u_{t}|^{2}\,dx=0$
and $\dis\int_{M}\langle\nabla p,u_{t}\rangle\,dx=0$. Then using
equation \eqref{eq2.1}, we get 
\[
\frac{1}{2}\frac{d}{dt}\int_{M}|u_{t}|^{2}\,dx+\nu\int_{M}\langle\square u_{t},u_{t}\rangle\,dx=0.
\]
Now using Bochner-Weitzenböck formula \eqref{Weitzenbock1} and \eqref{Weitzenbock2}
yields \eqref{eq2.15}. \end{proof}

\begin{proposition}\label{prop2.8} Assume that there exists a constant
$\kappa\in\R$ such that 
\begin{equation}
\Ric\geq-\kappa.\label{eq2.16}
\end{equation}
Then the following a priori estimate holds 
\begin{equation}
\frac{1}{2}||u_{t}||_{2}^{2}+\nu\int_{0}^{t}||\nabla u_{s}||_{2}^{2}\,ds\leq\frac{1}{2}||u_{0}||_{2}^{2}\,\exp({2\nu t\kappa^{+}}),\label{eq2.17}
\end{equation}
where $\kappa^{+}=\sup\{\kappa,0\}$. \end{proposition}

\begin{proof} Using \eqref{eq2.16} and \eqref{eq2.15}, we get inequality
\[
\frac{1}{2}\frac{d}{dt}\int_{M}|u_{t}|^{2}\,dx+\nu\int_{M}|\nabla u_{t}|^{2}\,dx\leq\nu\kappa\int_{M}|u_{t}|^{2}\,dx\leq\nu\kappa^{+}\int_{M}|u_{t}|^{2}\,dx.
\]
Let $\dis\psi(t)=\frac{1}{2}||u_{t}||_{2}^{2}+\nu\int_{0}^{t}||\nabla u_{s}||_{2}^{2}\,ds$.
Then $\psi$ satisfies inequality

\[
\psi(t)\leq\frac{1}{2}||u_{0}||_{2}^{2}+2\nu\kappa^{+}\,\int_{0}^{t}\psi(s)\,ds.
\]
Gronwall lemma yields \eqref{eq2.17}. \end{proof}

In what follows, we will establish the existence of weak solutions
in Leray sense over any $[0,T]$ and 
\[
u\in L^{2}([0,T],H^{1}(M))\cap L^{\infty}([0,T],L^{2}(M)).
\]

To this end, we will use the heat semi-group $\TT_{t}=e^{-t\square/2}$
to regularize vector fields. Let $v$ be a continuous vector field
on $M$ with compact support and define $\dis\TT_{t}v=(\TT_{t}\tilde{v})^{\#}$.
Then $\TT_{t}v$ solves the heat equation

\[
\left(\frac{\partial}{\partial t}+\frac{1}{2}\square\right)(\TT_{t}v)=0.
\]

By ellipticity of $\square$ (see for example \cite{Warner}), $(t,x)\ra(\TT_{t}v)(x)$
is smooth. It was shown in \cite{FLL} that 
\[
\div(\TT_{t}v)=\TT_{t}^{M}(\div(v)),
\]
where $\TT_{t}^{M}$ denotes heat semi-group on functions. Hence $\TT_{t}$
preserves the space of divergence free vector fields. By \eqref{eq6.6}
in Section \ref{sect6} it holds true that

\begin{equation}
|\TT_{t}v|\leq e^{t\kappa+/2}\ \TT_{t}^{M}|v|.\label{eq2.17-1}
\end{equation}
It follows that for $1\leq p\leq+\infty$, $\dis||\TT_{t}v||_{p}\leq e^{t\kappa+/2}\,||v||_{p}$,
and for $1\leq p<+\infty$, $\TT_{t}v\ra v$ in $L^{p}$.

\vskip 3mm Consider a family of smooth functions $\varphi_{\eps}\in C_{c}^{\infty}(M)$
with compact support such that 
\begin{equation}
0\leq\varphi_{\eps}\leq1,\quad\varphi_{\eps}(x)=1\ \hbox{for}\ x\in B(x_{M},1/\eps)\quad\hbox{and}\quad\sup_{\eps>0}||\nabla\varphi_{\eps}||_{\infty}<+\infty,\label{eq2.17-2}
\end{equation}
where $x_{M}$ is a fixed point of $M$. For $\eps>0$, we define
\[
F_{\eps}(u)=-\TT_{\eps}\PP\bigl(\varphi_{\eps}\,\nabla_{\TT_{\eps}u}(\varphi_{\eps}\TT_{\eps}u)\bigr)-\nu\TT_{\eps}\square\TT_{\eps}u,\quad u\in L^{2}(M)
\]
where $\PP$ is the orthogonal projection from $L^{2}(M)$ to the
subspace of vector fields of divergence free. We have

\[
||\TT_{\eps}\PP\bigl(\varphi_{\eps}\,\nabla_{\TT_{\eps}u}(\varphi_{\eps}\TT_{\eps}u)\bigr)||_{2}\leq e^{\eps\kappa^{+}/2}||\PP\bigl(\varphi_{\eps}\,\nabla_{\TT_{\eps}u}(\varphi_{\eps}\TT_{\eps}u)\bigr)||_{2}\leq e^{\eps\kappa^{+}/2}||\nabla_{\varphi_{\eps}\TT_{\eps}u}(\varphi_{\eps}\TT_{\eps}u)||_{2}.
\]

Since $\varphi_{\eps}$ is of compact support, we have 
\begin{equation}
||\nabla_{\varphi_{\eps}\TT_{\eps}u}(\varphi_{\eps}\TT_{\eps}u)||_{2}\leq||\varphi_{\eps}\TT_{\eps}u||_{\infty}\ ||\nabla(\varphi_{\eps}\TT_{\eps}u)||_{2}.\label{eq2.17-3}
\end{equation}

Again due to compact support of $\varphi_{\eps}$, when $n=3$, by
Sobolev's embedding theorem, there is a constant $\beta(\eps)>0$
such that 
\[
||\varphi_{\eps}\TT_{\eps}u||_{\infty}\leq\beta(\eps)\,||\varphi_{\eps}\TT_{\eps}u||_{H^{2}}.
\]

For the general case, it is sufficient to bound the uniform norm by
the norm of $H^{m}$ with $m>\frac{n}{2}$.

\begin{proposition}\label{prop2.9} For any $T>0$, there are constants
$\beta_{1},\beta_{2}$ such that 
\begin{equation}
||\square\TT_{\eps}u||_{2}\leq\frac{\beta_{1}}{\eps}||u||_{2},\quad||\nabla\TT_{\eps}u||_{2}\leq\frac{\beta_{2}}{\sqrt{\eps}},\quad\eps>0.\label{eq2.17-4}
\end{equation}
\end{proposition}

\begin{proof} We will restate, in Section \ref{sect6}, \eqref{eq2.17-4}
with more precise coefficients dependent of curvatures of $M$ and
give a proof based on Bismut formulae obtained in \cite{ElworthyLi,DriverTh}.
\end{proof}

By Proposition \ref{prop2.9}, there are constants $\beta(\eps)>0,\tilde{\beta}(\eps)>0$
such that

\begin{equation}
||\varphi_{\eps}\TT_{\eps}u||_{\infty}\leq\beta(\eps)\,||u||_{2},\quad||\TT_{\eps}\square\TT_{\eps}u||_{2}\leq\tilde{\beta}(\eps)\,||u||_{2}.\label{eq2.17-5}
\end{equation}

Combining \eqref{eq2.17-3} and \eqref{eq2.17-5}, there are two constants
$\beta_{1}(\eps)>0$ and $\beta_{2}(\eps)>0$ such that 
\[
||F_{\eps}(u)||_{2}\leq\beta_{1}(\eps)\,||u||_{2}^{2}+\beta_{2}(\eps)||u||_{2},
\]
and $F_{\eps}$ is locally Lipschitz. By theory of ordinary differential
equation, there is a unique solution $u^{\eps}$ to 
\begin{equation}
\frac{du_{t}^{\eps}}{dt}=F_{\eps}(u_{t}^{\eps}),\quad u_{0}^{\eps}=u_{0}\in L^{2},\quad\div(u_{t}^{\eps})=0,\label{eq2.17-6}
\end{equation}
up to the explosion time $\tau$.

\begin{theorem}\label{th1} Assume that $||\Ric||_{\infty}<+\infty$
and that ${\mathcal{R}}_{2}$ is bounded below. Then for any $T>0$,
there is a weak solution $u\in L^{2}([0,T],H^{1})$ to Navier-Stokes
equation \eqref{eq2.1} such that 
\[
\frac{1}{2}||u_{t}||_{2}^{2}+\nu\int_{0}^{t}||\nabla u_{s}||_{2}^{2}\,ds\leq\frac{1}{2}||u_{0}||_{2}^{2}\,\exp({2\nu t\kappa^{+}}),
\]
where $\kappa$ is lower bound of $\Ric$ and ${\mathcal{R}}_{2}$
is the Weitzenböck curvature on $2$-differential forms defined in
\eqref{eq6.2}. \end{theorem}

\begin{proof} Rewriting Equation \eqref{eq2.17-6} in the following
explicit form, for $t<\tau$, 
\[
\frac{du_{t}^{\eps}}{dt}+\TT_{\eps}\PP\bigl(\varphi_{\eps}\,\nabla_{\TT_{\eps}u_{t}^{\eps}}(\varphi_{\eps}\TT_{\eps}u_{t}^{\eps})\bigr)+\nu\TT_{\eps}\square\TT_{\eps}u_{t}^{\eps}=0.
\]
Note that 
\[
\begin{split}\int_{M}\langle\TT_{\eps}\PP\bigl(\varphi_{\eps}\,\nabla_{\TT_{\eps}u_{t}^{\eps}}(\varphi_{\eps}\TT_{\eps}u_{t}^{\eps})\bigr),\ u_{t}^{\eps}\rangle\,dx & =\int_{M}\langle\nabla_{\TT_{\eps}u_{t}^{\eps}}(\varphi_{\eps}\TT_{\eps}u_{t}^{\eps})\bigr),\ \varphi_{\eps}\TT_{\eps}u_{t}^{\eps}\rangle\,dx\\
 & =\int_{M}\L_{\TT_{\eps}u_{t}^{\eps}}|\varphi_{\eps}\TT_{\eps}u_{t}^{\eps}|^{2}\ dx=0,
\end{split}
\]
since $\div(\TT_{\eps}u_{t}^{\eps})=0$, and

\[
\int_{M}\langle\TT_{\eps}\square\TT_{\eps}u_{t}^{\eps},\ u_{t}^{\eps}\rangle\ dx=\int_{M}|\nabla\TT_{\eps}u_{t}^{\eps}|^{2}\,dx+\int_{M}\langle\Ric(\TT_{\eps}u_{t}^{\eps}),\ \TT_{\eps}u_{t}^{\eps}\rangle\ dx.
\]
Hence 
\[
\begin{split}\frac{1}{2}\frac{d}{dt}\int_{M}|u_{t}^{\eps}|^{2}\,dx+\nu\int_{M}|\nabla\TT_{\eps}u_{t}^{\eps}|^{2}\,dx & =-\nu\int_{M}\langle\Ric(\TT_{\eps}u_{t}^{\eps}),\ \TT_{\eps}u_{t}^{\eps}\rangle\ dx\\
 & \leq-\nu\kappa\int_{M}|\TT_{\eps}u_{t}^{\eps}|^{2}\,dx,
\end{split}
\]
or in the form 
\begin{equation}
\frac{1}{2}||u_{t}^{\eps}||_{2}^{2}+\nu\int_{0}^{t}|||\nabla\TT_{\eps}u_{s}^{\eps}||_{2}^{2}\,ds\leq\frac{1}{2}||u_{0}||_{2}^{2}+\nu\kappa^{+}\int_{0}^{t}||\TT_{\eps}u_{s}^{\eps}||_{2}^{2}\,ds.\label{eq2.17-7}
\end{equation}

According to \eqref{eq2.17-1}, above inequality implies that
\[
\frac{1}{2}||u_{t}^{\eps}||_{2}^{2}\leq\frac{1}{2}||u_{0}||_{2}^{2}+\nu\kappa^{+}e^{\eps\kappa^{+}}\int_{0}^{t}||u_{s}^{\eps}||_{2}^{2}\,ds.
\]
Gronwall lemma implies that for $t<\tau$

\[
\frac{1}{2}||u_{t}^{\eps}||_{2}^{2}\leq\frac{1}{2}||u_{0}||_{2}^{2}\,\exp(t\nu\kappa^{+}e^{\eps\kappa^{+}}).
\]
It follows that $\tau=+\infty$. Now again by \eqref{eq2.17-1} and
\eqref{eq2.17-7}, we get

\[
\frac{1}{2}||\TT_{\eps}u_{t}^{\eps}||_{2}^{2}+\nu e^{\eps\kappa^{+}}\int_{0}^{t}|||\nabla\TT_{\eps}u_{s}^{\eps}||_{2}^{2}\,ds\leq\frac{1}{2}e^{\eps\kappa^{+}}||u_{0}||_{2}^{2}+\nu\kappa^{+}e^{\eps\kappa^{+}}\int_{0}^{t}||\TT_{\eps}u_{s}^{\eps}||_{2}^{2}\,ds.
\]

Gronwall lemma yields, for $\eps\leq1$, that

\begin{equation}
\frac{1}{2}||\TT_{\eps}u_{t}^{\eps}||_{2}^{2}+\nu e^{\eps\kappa^{+}}\int_{0}^{t}|||\nabla\TT_{\eps}u_{s}^{\eps}||_{2}^{2}\,ds\leq\frac{e^{\kappa^{+}}}{2}||u_{0}||_{2}^{2}\,\exp(t\nu\kappa^{+}e^{\kappa^{+}}).\label{eq2.17-8}
\end{equation}

Let $T>0$. By \eqref{eq2.17-8}, the family $\bigl\{\TT_{\eps}u_{\cdot}^{\eps};\ \eps\in(0,1]\bigr\}$
is bounded in $L^{2}([0,T],H^{1})$ as well in $L^{\infty}([0,T],L^{2})$.
Then there is a sequence $\eps_{n}$ and a $u\in L^{2}([0,T],H^{1})\cap L^{\infty}([0,T],L^{2})$
such that $\TT_{\eps_{n}}u^{\eps_{n}}$ converges weakly to $u$ in
$L^{2}([0,T],H^{1})$ and $*$-weakly in $L^{\infty}([0,T],L^{2})$.
Now standard arguments allow to prove that $u$ is a weak solution
\eqref{eq2.1}. The boundedness of $\Ric$ is needed while passing
to the limit of the term $\dis\int_{M}\langle\Ric(\TT_{\eps}u_{t}^{\eps}),\ v_{t}\rangle\ dx$.

\end{proof}

\begin{proposition}\label{prop2.10} Let $\dim(M)=3$. The vorticity
$\omega_{t}$ satisfies a priori identity: 
\begin{equation}
\frac{1}{2}\frac{d}{dt}\int_{M}|\omega_{t}|^{2}\,dx+\nu\int_{M}|\nabla\omega_{t}|^{2}\,dx=-\nu\int_{M}\langle\Ric\,\omega_{t},\omega_{t}\rangle\,dx+\int_{M}\langle\omega_{t}\intt\nabla^{s}u_{t},\omega_{t}\rangle\,dx.\label{eq2.18}
\end{equation}
\end{proposition}

\begin{proof} Using Equation \eqref{eq2.11} and the same as proving
\eqref{eq2.15} yields \eqref{eq2.18}. \end{proof}

\vskip 3mm The term $H_{t}:=\int_{M}(\omega_{t},u_{t})\,dx$ is called
helicity in theory of the fluid mechanics.

\begin{proposition}\label{prop2.11} Let $\dim(M)=3$. Then 
\begin{equation}
\begin{split}\frac{d}{dt}\int_{M}(\omega_{t},u_{t})\,dx= & -\nu\int_{M}\langle d\omega_{t},\ *\omega_{t}\rangle_{\Lambda^{2}}\,dx-\nu\int_{M}(\nabla\omega_{t},\nabla u_{t})\,dx\\
 & -\nu\int_{M}(\omega_{t},\Ric\,u_{t})\,dx+\int_{M}(\omega_{t},\nabla_{u_{t}}^{s}u_{t})\,dx.
\end{split}
\label{eq2.19}
\end{equation}
\end{proposition}

\begin{proof} Using Equation \eqref{eq2.1} and Equation \eqref{eq2.11},
we have 
\[
\begin{split}\frac{d}{dt}(\omega_{t},u_{t}) & =-(\nabla_{u_{t}}\omega_{t},u_{t})-\nu(\square\omega_{t},u_{t})+(\omega_{t}\intt\nabla^{s}u_{t},u_{t})\\
 & -(\omega_{t},\nabla_{u_{t}}u_{t})-\nu(\omega_{t},\square u_{t})-(\omega_{t},\nabla p).
\end{split}
\]
It is obvious that 
\[
\dis\int_{M}\Bigl[(\nabla_{u_{t}}\omega_{t},u_{t})+(\omega_{t},\nabla_{u_{t}}u_{t})\Bigr]\,dx=\int_{M}\L_{u_{t}}(\omega_{t},u_{t})\,dx=0.
\]
In addition, by (\cite{Warner}, page 220), $\dis d^{\ast}=(-1)^{n(p+1)+1}*d*$
and $**=(-1)^{p(n-p)}$ on $p$-forms. Then $d^{\ast}*=\pm\ *d$,
so that

\[
\int_{M}\langle\omega_{t},dp\rangle\,dx=\int_{M}\langle*\tilde{\omega}_{t},dp\rangle\,dx=\int_{M}d^{\ast}(*\tilde{\omega}_{t})p\,dx=\pm\int_{M}*(d\tilde{\omega}_{t})\,p\,dx=0.
\]

On one hand, using Hodge star operator, 
\[
\int_{M}(\omega_{t},\square u_{t})\,dx=\int_{M}\langle\omega_{t},d^{\ast}d\tilde{u}_{t}\rangle\,dx=\int_{M}\langle d\omega_{t},\tilde{\omega}_{t}\rangle\,dx=\int_{M}\langle d\omega_{t},*\omega_{t}\rangle\,dx.
\]
On the other hand, using Bochner-Weitzenböck formula, 
\[
\int_{M}(\omega_{t},\square u_{t})\,dx=\int_{M}(\nabla\omega_{t},\nabla u_{t})\,dx+\int_{M}(\omega_{t},\Ric\,u_{t})dx.
\]

By putting these terms together we conclude that 
\[
\begin{split}\frac{d}{dt}\int_{M}(\omega_{t},u_{t})\,dx= & -\nu\int_{M}\langle d\omega_{t},*\omega_{t}\rangle\,dx-\nu\int_{M}(\nabla\omega_{t},\nabla u_{t})\,dx\\
 & -\nu\int_{M}(\omega_{t},\Ric\,u_{t})dx+\int_{M}(\omega_{t},\nabla_{u_{t}}^{s}u_{t})\,dx,
\end{split}
\]
since $(\omega_{t}\intt\nabla^{s}u_{t},u_{t})=(\omega_{t},\nabla_{u_{t}}^{s}u_{t})$.
We get \eqref{eq2.19}. \end{proof}

\section{Heat equations on differential forms}

\label{sect3}

We will express solutions to equation \eqref{eq2.11} by means of
principal bundle of orthonormal frames $O(M)$. An element $r\in O(M)$
is an isometry from $\R^{n}$ onto $T_{\pi(r)}M$ where $\pi:O(M)\ra M$
is the canonical projection. More precisely, an element of $O(M)$
is composed of $(x,r)$, where $x=\pi(x,r)$ and $r$ is an orthonormal
frame at $x$, that is, an isometry from $\R^{n}$ onto $T_{x}M$.
For the sake of simplicity, we read $r$ as $(\pi(r),r)$, but we
sometimes have to distinguish them. The Levi-Civita connection on
$M$ gives rise to $n$ canonical horizontal vector fields $\{A_{1},\ldots,A_{n}\}$
on $O(M)$, which are such that $d\pi(r)\cdot A_{r}=r\eps_{i}$, where
$\{\eps_{1},\ldots,\eps_{n}\}$ is the canonical basis of $\R^{n}$.
A vector field $v$ on $M$ can be lift to a horizontal vector field
$V$ on $O(M)$ such that $d\pi(r)V_{r}=v_{\pi(r)}$. Let $\omega$
be a differential $1$-form. Following Malliavin \cite{Malliavin},
we define 
\begin{equation}
F_{\omega}^{i}(r)=(\omega_{\pi(r)},r\eps_{i})=(\pi^{\ast}\omega,A_{i})_{r},\quad i=1,\ldots,n,\label{eq3.0}
\end{equation}
where $\pi^{\ast}\omega$ is the pull-back of $\omega$ by $\pi:O(M)\ra M$.
We have 
\begin{equation}
(\L_{A_{j}}F_{\omega}^{i})(r)=(\nabla_{r\eps_{j}}\omega,r\eps_{i})=(\nabla\omega,r\eps_{j}\otimes r\eps_{i}),\label{eq3.1}
\end{equation}
where the second duality makes sense in $T_{\pi(r)}M\otimes T_{\pi(r)}M$.
In fact, let $t\ra r(t)\in O(M)$ be the smooth curve such that $r(0)=r,r'(0)=A_{j}(r)$.
Let $\xi_{t}=\pi(r(t))$. Then $//_{t}^{-1}:=r\circ r(t)^{-1}$ is
the parallel translation from $T_{\xi_{t}}M$ onto $T_{x}M$ along
$\xi_{\cdot}$ and
\[
F_{\omega}^{i}(r(t))=(\omega_{\xi_{t}},r(t)\eps_{i})=(//_{t}^{-1}\omega_{\xi_{t}},r\eps_{i}).
\]
Taking the derivative with respect to $t$ at $t=0$ yields \eqref{eq3.1}.
In the same way, we get $(\L_{A_{j}}^{2}F_{\omega}^{i})(r)=(\nabla_{r\eps_{j}}\nabla\omega,r\eps_{j}\otimes r\eps_{i})$.
Therefore

\[
\Delta_{O(M)}F_{\omega}^{i}:=\sum_{j=1}^{n}\L_{A_{j}}^{2}F_{\omega}^{i}=(\Delta\omega,r\eps_{i})=F_{\Delta\omega}^{i}(r).
\]

Let $U_{t}$ be the horizontal lift of $u_{t}$ to $O(M)$. Then $\dis U_{t}(r)=\sum_{j=1}^{n}\langle u_{t}(x),r\eps_{j}\rangle A_{j}(r)$,
where $x=\pi(r)$ and according to \eqref{eq3.1},

\[
(\L_{U_{t}}F_{\omega}^{i})(r)=\sum_{j=1}^{n}\langle u_{t},r\eps_{j}\rangle(\L_{A_{j}}F_{\omega}^{i})(r)=\langle\nabla_{u_{t}}\omega,r\eps_{i}\rangle=F_{\nabla_{u_{t}}\omega}^{i}(r).
\]

Let $\phi_{t}=\omega_{t}\intt\nabla^{s}u$; then 
\[
F_{\phi_{t}}^{i}(r)=(\phi_{t},r\eps_{i})=\omega_{t}(\nabla_{r\eps_{i}}^{s}u_{t})=\sum_{j=1}^{n}\langle\nabla_{r\eps_{i}}^{s}u_{t},r\eps_{j}\rangle\,(\omega_{t},r\eps_{j})=\sum_{j=1}^{n}\langle\nabla_{r\eps_{i}}^{s}u_{t},r\eps_{j}\rangle F_{\omega_{t}}^{j}.
\]

Define $\dis K_{ij}(t,r)=\langle\nabla_{r\eps_{i}}^{s}u_{t}(\pi(r)),r\eps_{j}\rangle$
and $K(t,r)=(K_{ij}(t,r))$. Then $F_{\phi_{t}}(r)=K(t,r)F_{\omega_{t}}(r)$.
By applying Bochner-Weitzenböck formula (see \eqref{Weitzenbock1})
to $1$-form $\omega$, $\square\omega=-\Delta\omega+\Ric^{\#}\omega$.
Let $\ric_{r}=r^{-1}\Ric_{\pi(r)}r$ denote the equivariant representation
of $\Ric$ on $O(M)$. Then $\dis F_{\Ric^{\#}\omega}=\ric\,F_{\omega}$,
since $\ric$ is symmetric. Now applying $F$ on two sides of Equation
\eqref{eq2.11}, we get the following heat equation defined on $O(M)$,
but taking values in flat space $\R^{n}$:

\begin{equation}
\frac{d}{dt}F_{\omega_{t}}=\nu\Delta_{O(M)}F_{\omega_{t}}-\L_{U_{t}}F_{\omega_{t}}+(K(t,\cdot)-\nu\,\ric)F_{\omega_{t}}.\label{eq3.2}
\end{equation}

This equation was extensively studied in the field of Stochastic analysis,
see \cite{Bakry,Bismut1,Bismut2,Elworthy,ELL,IW,Kunita,Malliavin,Stroock}
for example. However the situation becomes more complicated when the
vector field is time-dependent (see \cite{SW}).

\vskip 3mm In what follows, we will derive a stochastic representation
formula for the solution to \eqref{eq3.2}. First of all, we have
to prove that the concerned diffusion processes do not explode at
a finite time. For this purpose, consider a family of vector fields
$\{v_{t}(x);\ t\geq0\}$ on $M$. We assume here that $(t,x)\ra v_{t}(x)$
is continuous and for each $t\geq0$, $v_{t}\in C^{1+\alpha}$ with
$\alpha>0$, and $\div(v_{t})=0$. Let $V_{t}$ be the horizontal
lift of $v_{t}$ to $O(M)$. Then $\div(V_{t})=\div(v_{t})\circ\pi$
(see \cite{FLL}, 595) and therefore $\div(V_{t})=0$. \vskip 3mm

Consider the following Stratonovich stochastic differential equation
(SDE) 
\begin{equation}
dr_{t}=\sum_{k=1}^{n}A_{k}(r_{t})\circ dW_{t}^{k}+V_{t}(r_{t})dt,\quad r_{|_{t=0}}=r_{0}\label{eq3.2-1}
\end{equation}
where $W_{t}=(W_{t}^{1},\cdots,W_{t}^{n})$ is a standard Brownian
motion on $\R^{n}$. Denote by $r_{t}(w,r_{0})$ the solution to \eqref{eq3.2-1}.
Let $\zeta(w,r_{0})$ be the explosion time of SDE \eqref{eq3.2-1}.
Let

\[
\Sigma(t,w)=\{r_{0}\in O(M);\ \zeta(w,r_{0})>t\}.
\]

Then for each $t>0$ given, almost surely $\Sigma(t,w)$ is an open
subset of $O(M)$ and $r_{0}\ra r_{t}(w,r_{0})$ is a local diffeomorphism
on $\Sigma(t,w)$ (see \cite{Kunita}). To be short, set $r_{t}(r_{0})=r_{t}(w,r_{0})$.
The Jacobian $J_{r_{t}}$ of $r_{0}\ra r_{t}(r_{0})$ is equal to
$1$, since by \cite{Kunita}, the Jacobian $J_{r_{t}^{-1}}$ of inverse
map $r_{t}^{-1}$ admits expression

\[
J_{r_{t}^{-1}}=\exp\Bigl(-\int_{0}^{t}\sum_{k=1}^{n}\div(H_{k})(r_{s}(r_{0}))\circ dW_{s}^{k}-\int_{0}^{t}\div(V_{s}(r_{s}(r_{0}))\,ds\Bigr)=1.
\]

Then for any $\varphi\in C_{c}(O(M))$, almost surely, 
\begin{equation}
\int_{O(M)}\varphi(r_{t}(r_{0}))\,{\bf 1}_{\Sigma(t,w)}(r_{0})\,dr_{0}=\int_{O(M)}\varphi(r_{0}){\bf 1}_{r_{t}(\Sigma(t,w))}(r_{0})\ dr_{0},\label{eq3.2-2}
\end{equation}
where $dr_{0}$ is the Liouville measure on $O(M)$ (see \cite{Stroock},
page 185) such that $\pi_{\#}(dr_{0})=dx_{0}$.

\vskip 3mm Let $d_{M}(x,y)$ be the Riemannian distance on $M$ between
$x$ and $y$. Fix a reference point $x_{M}\in M$, consider

\[
\rho(r)=d_{M}(\pi(r),x_{M}).
\]

It is known that for each $x_{0}$ given, $x\ra d_{M}(x,x_{0})$ is
smooth out of $C_{x_{0}}\cup\{x_{0}\}$, were $C_{x_{0}}$ is the
cut-locusof $x_{0}$. It is known that $C_{x_{0}}$ is negligible
with respect to $dx$. Then $\rho$ is smooth out of $\pi^{-1}(C_{x_{M}}\cup\{x_{M}\})$.
By \cite{Stroock}, p. 197, out of $\pi^{-1}(C_{x_{0}}\cup\{x_{0}\})$,
\begin{equation}
\frac{1}{2}\Delta_{O(M)}d_{M}(\pi(\cdot),x_{0})\leq\frac{n-1}{2d_{M}(\pi(\cdot),x_{0})}+\frac{1}{2}\sqrt{n\kappa}.\label{eq3.2-3}
\end{equation}

It is known that out of $C_{x_{0}}\cup\{x_{0}\}$, $|\nabla_{x}d_{M}(x,x_{0})|=1$.
Therefore out of $\pi^{-1}(C_{x_{0}}\cup\{x_{0}\})$, 
\begin{equation}
|\L_{V_{t}}d_{M}(\pi(\cdot),x_{0})|\leq|V_{t}|.\label{eq3.2-4}
\end{equation}

The lower bound of $\frac{1}{2}\Delta_{O(M)}\rho$ is more delicate.
According to \cite{Hsu}, page 90,

\begin{equation}
\frac{1}{2}\Delta_{O(M)}d_{M}(\pi(\cdot),x_{0})\geq\frac{n-1}{2\rho}-\frac{1}{2}\sqrt{n(n-1)K_{R}},\ quad\pi(r)\in B(x_{M},R)\backslash(C_{x_{0}}\cup\{x_{0}\}).\label{eq3.2-5}
\end{equation}
where $K_{R}$ is the upper bound of sectional curvature on the big
ball $B(x_{M},R)$.

\begin{proposition}\label{prop3.1} Assume furthermore that 
\begin{equation}
\int_{0}^{T}\!\!\int_{M}|v_{s}(x)|^{2}\,dxds<+\infty.\label{eq3.2-5-1}
\end{equation}
Then there is a non-decreasing process $\hat{L}_{t}\geq0$ and a Brownian
motion $\{\beta_{t};\ t\geq0\}$ on $\R$ such that for almost surely
initial $r_{0}$,

\begin{equation}
\rho(r_{t})-\rho(r_{0})=\beta_{t}+\int_{0}^{t}\bigl((\frac{1}{2}\Delta_{O(M)}+\L_{V_{s}})\rho\bigr)(r_{s})\ ds-\hat{L}_{t},\quad t<\zeta(w,r_{0}).\label{eq3.2-6}
\end{equation}
\end{proposition}

\begin{proof} The proof will be given in Section \ref{sect6}. \end{proof}

\begin{theorem}\label{th3.2} Assume $\Ric\geq-\kappa$ and \eqref{eq3.2-5-1}
holds. Then for almost all $r_{0}$, $\zeta(w,r_{0})=+\infty$ almost
surely. \end{theorem} 

\begin{proof} We have, by \eqref{eq3.2-6}, 
\[
\rho(r_{t\wedge\zeta})^{2}\leq\rho(r_{0})^{2}+t\wedge\zeta+2\int_{0}^{t\wedge\zeta}\rho(r_{s})d\beta_{s}+2\int_{0}^{t\wedge\zeta}\rho(r_{s})\,(L_{s}\rho)(r_{s})\,ds,
\]
where $\L_{s}=\frac{1}{2}\Delta_{O(M)}+\L_{V_{s}}$. Using \eqref{eq3.2-3}
and \eqref{eq3.2-4}, there is constants $C>0$ such that 
\[
\begin{split} & \E(\rho(r_{t\wedge\zeta})^{2})\leq\rho(r_{0})^{2}+C\int_{0}^{t}\E\Bigl(\bigl(2\rho(r_{s})(L_{s}\rho)(r_{s})+1\bigr){\bf 1}_{(s<\zeta)}\Bigr)\,ds\\
 & \leq\rho(r_{0})^{2}+2C\int_{0}^{t}\E\Bigl((1+\rho(r_{s}))(1+|V_{s}(r_{s})|){\bf 1}_{(s<\zeta)}\Bigr)\,ds.
\end{split}
\]
Let $\mu$ be the probability measure on $O(M)$ defined in \eqref{measure}.
Then

\[
\begin{split}\int_{O(M)}\E(\rho(r_{t\wedge\zeta})^{2})\,d\mu & \leq\int_{O(M)}\rho(r_{0})^{2}d\mu+2C\int_{0}^{t}\!\!\int_{O(M)}\E\Bigl((1+\rho(r_{s}))(1+|V_{s}(r_{s})|){\bf 1}_{(s<\zeta)}\Bigr)\,d\mu ds\\
 & \leq\int_{O(M)}\rho(r_{0})^{2}d\mu+4C\Bigl(\int_{0}^{t}\!\!\int_{O(M)}\E\Bigl((1+\rho(r_{s\wedge\zeta})^{2})\Bigr)\,d\mu ds\Bigr)^{1/2}\times\\
 & \hskip20mm\times\Bigl(\int_{0}^{t}\!\!\int_{O(M)}\E\Bigl((1+|V_{s}(r_{s})|)^{2}{\bf 1}_{(s<\zeta)}\Bigr)\,d\mu ds\Bigr)^{1/2}.
\end{split}
\]
Note that 
\[
\int_{0}^{t}\!\!\int_{O(M)}\E\Bigl((1+|V_{s}(r_{s})|)^{2}{\bf 1}_{(s<\zeta)}\Bigr)\,d\mu ds\leq2\Bigl(T+\int_{0}^{T}\!\!\int_{M}|v_{s}(x)|^{2}\,dxds\Bigr).
\]

Set $\dis\psi(t)=\int_{O(M)}\E(\rho(r_{t\wedge\zeta})^{2})\,d\mu$
and

\begin{equation}
C(T,v)=4C\sqrt{2}\sqrt{T+||v||_{L^{2}([0,T]\times M)}^{2}}.\label{eq3.2-12}
\end{equation}

Remarking that $\sqrt{\xi}\leq1+\xi$ for $\xi\geq0$, above two inequalities
imply that 
\[
\psi(t)\leq\Bigl(\int_{O(M)}\rho(r_{0})^{2}\,d\mu+C(T,v)\Bigr)+C(T,v)\int_{0}^{t}\psi(s)\,ds.
\]
The Gronwall lemma then yields 
\[
\int_{O(M)}\E(\rho(r_{t\wedge\zeta})^{2})\,d\mu\ \leq\Bigl(\int_{O(M)}\rho(r_{0})^{2}\,d\mu+C(T,v)\Bigr)\exp(C(T,v)).
\]
The result follows. \end{proof}

\vskip 3mm Now we are going to obtain a probabilistic representation
for solution to the heat equation \eqref{eq3.2}. To this end, set
$F(t,r)=F_{\omega_{t}}(r)$. Let $T>0$ be fixed. Assume that $u_{t}$
is a solution to \eqref{eq2.1} such that $(t,x)\ra u_{t}(x)$ is
continuous and for each $t\geq0$, $u_{t}\in C^{1+\alpha}$ with $\alpha>0$.
Consider the following SDE on $O(M)$,

\begin{equation}
\begin{cases}
dr_{s,t}(r,w)=\sqrt{2\nu}\dis\sum_{i=1}^{n}A_{i}(r_{s,t}(r,w))\circ dW_{t}^{i}-U_{T-t}(r_{s,t}(r,w))\,dt,\quad s<t<T,\\
r_{s,s}(r,w)=r.
\end{cases}\label{eq3.3}
\end{equation}

Let $v_{t}(x)=u_{T-t}(x)$. Then by Theorem \ref{th3.2}, SDE \eqref{eq3.3}
is stochastic complete. Let $Q_{s,t}(w)$ be solution to the resolvent
equation

\begin{equation}
\frac{d}{dt}Q_{s,t}(w)=Q_{s,t}(w)J_{T-t}(r_{s,t}(r,w)),\quad s<t<T,\ Q_{s,s}(w)=Id\label{eq3.4}
\end{equation}
where 
\begin{equation}
J_{t}(r)=K(t,r)-\nu\,\ric_{r}.\label{eq3.4-5}
\end{equation}

For the sake of simplicity, we denote $r_{s,t}=r_{s,t}(r,w)$. Applying
Itô formula to $Q_{s,t}F(T-t,r_{s,t})$ for $d_{t}$ with $t\in(s,T)$,
we have

\[
\begin{split} & d_{t}\Bigl(Q_{s,t}\,F(T-t,r_{s,t})\Bigr)=d_{t}Q_{s,t}\,F(T-t,r_{s,t})+Q_{s,t}\,d_{t}\bigl(F(T-t,r_{s,t})\bigr)\\
 & =Q_{s,t}J_{T-t}(r_{s,t})F(T-t,r_{s,t})+\sqrt{2\nu}\,Q_{s,t}\,\sum_{i=1}^{n}(\L_{A_{i}}F)(T-t,r_{s,t})\,dW_{t}^{i}\\
 & +Q_{s,t}\Bigl(-(\partial_{t}F)(T-t,r_{s,t})+\nu\,(\Delta_{O(M)}F)(T-t,r_{s,t})-(\L_{U_{T-t}}F)(T-t,r_{s,t})\Bigr)\,dt\\
 & =\sqrt{2\nu}\,Q_{s,t}\,\sum_{i=1}^{n}(\L_{A_{i}}F)(T-t,r_{s,t})\,dW_{t}^{i},
\end{split}
\]
where the last equality is due to Equation \eqref{eq3.2}. It follows
that 
\[
Q_{s,t}\,F(T-t,r_{s,t})-F(T-s,r)=\sqrt{2\nu}\,\sum_{i=1}^{n}\int_{s}^{t}Q_{s,\tau}\,(\L_{A_{i}}F)(T-\tau,r_{s,\tau})\,dW_{\tau}^{i}.
\]
Taking expectation on the two sides gives $\dis\E\Bigl(Q_{s,t}\,F(T-t,r_{s,t})\Bigr)=F(T-s,r)$.
Let $t=T$. Then $\dis\E\Bigl(Q_{s,T}\,F(0,r_{s,T})\Bigr)=F(T-s,r)$.
Replacing $s$ by $T-t$, we get the following representation formula
to \eqref{eq3.2}:

\begin{equation}
F_{\omega_{t}}=\E\Bigl(Q_{T-t,T}F_{\omega_{0}}(r_{T-t,T})\Bigr).\label{eq3.5}
\end{equation}

\vskip 3mm

In what follows, we will explain how a vector field $v$ on $M$ gives
rise to a metric compatible connection $\Gamma^{v}$. For a time-independent
vector field $v$ on $M$, the diffusion processes $\{x_{t},\ t\geq0\}$
associated to the generator $\frac{1}{2}\Delta_{M}+v$ can be constructed
in the following way:

\begin{equation}
dr_{t}=\sum_{i=1}^{n}A_{i}(r_{t})\circ dW_{i}^{i}+V(r_{t})\,dt\label{eq3.5-1}
\end{equation}
where $V$ is the horizontal lift of $v$ to $O(M)$, and let $x_{t}=\pi(r_{t})$.
We assume that the lift-time $\zeta=+\infty$ almost surely. \vskip
3mm

In Chapter V of \cite{IW}, Ikeda and Watanabe introduced a metric
compatible connection $\Gamma^{v}$ so that the diffusion process
of generator $\frac{1}{2}\Delta_{M}+v$ can be constructed by rolling
without friction Brownian motion on $\R^{n}$ with respect to the
connection $\Gamma^{v}$. More precisely let $\{B_{1},\ldots,B_{n}\}$
be the canonical horizontal vector fields on $O(M)$ with respect
to $\Gamma^{v}$, consider SDE on $O(M)$:

\[
dr_{w}(t)=\sum_{i=1}^{n}B_{i}(r_{w}(t))\circ dW_{t}^{i},\quad r_{w}(0)=r.
\]
Then the generator of diffusion process $t\ra x_{t}(w)=\pi(r_{w}(t))$
is $\frac{1}{2}\Delta_{M}+v$. In fact, it holds

\begin{equation}
\frac{1}{2}\sum_{j=1}^{n}\L_{B_{j}}^{2}(f\circ\pi)=\bigl((\frac{1}{2}\Delta_{M}+v)f\bigr)\circ\pi.\label{eq3.7}
\end{equation}
\vskip 3mm This connection $\Gamma^{v}$ was defined locally in \cite{IW}.
On a local chart $U$, $\{\frac{\partial}{\partial x_{1}},\ldots,\frac{\partial}{\partial x_{n}}\}$
is a local basis of tangent spaces $T_{x}M$ with $x\in U$, and $\dis v=\sum_{i=1}^{n}v^{i}\frac{\partial}{\partial x_{i}}$.
Let $\Gamma_{ij}^{0,k}$ be the Christoffel coefficients of Levi-Civita
connection. According to (\cite{IW}, p.271), the Christoffel coefficients
$\Gamma_{ij}^{k}$ of $\Gamma^{v}$ is defined by (see also \cite{AM}),

\begin{equation}
\Gamma_{ij}^{k}=\Gamma_{ij}^{0,k}-\frac{2}{n-1}\Bigl(\delta_{ki}\sum_{\ell=1}^{n}g_{j\ell}v^{\ell}-g_{ij}v^{k}\Bigr).\label{eq3.8}
\end{equation}

\begin{proposition}\label{prop3.1} Let $\nabla^{v}$ be the covariant
derivative with respect to the connection $\Gamma^{v}$, and $\nabla^{0}$
with respect to the Levi-Civita connection. Then for two vector fields
$X,Y$ on $M$, 
\begin{equation}
\nabla_{X}^{v}Y=\nabla_{X}^{0}Y-\frac{2}{n-1}K_{v}(X,Y),\label{eq3.9}
\end{equation}
where 
\begin{equation}
K_{v}(X,Y)=\langle Y,v\rangle\,X-\langle X,Y\rangle\,v.\label{eq3.10}
\end{equation}
\end{proposition}

\begin{proof} We have, using \eqref{eq3.8},

\[
\begin{split}\nabla_{X}^{v}Y & =\sum_{k=1}^{n}\Bigl[\sum_{i,j=1}^{n}X^{i}Y^{j}\Gamma_{ij}^{k}+\sum_{i=1}^{n}X^{i}\frac{\partial}{\partial x_{i}}Y^{k}\Bigr]\frac{\partial}{\partial{x_{k}}}\\
 & =\sum_{k=1}^{n}\Bigl[\sum_{i,j=1}^{n}X^{i}Y^{j}\Gamma_{ij}^{0,k}+\sum_{i=1}^{n}X^{i}\frac{\partial}{\partial x_{i}}Y^{k}\Bigr]\frac{\partial}{\partial{x_{k}}}-\frac{2}{n-1}I_{2},
\end{split}
\]
where 
\[
I_{2}=\sum_{i,j,k=1}^{n}X^{i}Y^{j}\delta_{ki}\langle\frac{\partial}{\partial x_{j}},v\rangle\frac{\partial}{\partial x_{k}}-\sum_{i,j,k=1}^{n}X^{i}Y^{j}\langle\frac{\partial}{\partial x_{i}},\frac{\partial}{\partial x_{j}}\rangle v^{k}\frac{\partial}{\partial x_{k}},
\]
since $\dis\sum_{\ell=1}^{n}g_{j\ell}v^{\ell}=\langle\frac{\partial}{\partial{x_{j}}},v\rangle$.
It is obvious to see that the first sum in $I_{2}$ is equal to $\langle Y,v\rangle\,X$,
while the second sum yields $\langle X,Y\rangle\,v$. The relation
\eqref{eq3.9} and \eqref{eq3.10} follow. \end{proof}

\vskip 3mm Having this explicit expression, we will compute the associated
torsion tensor $T^{v}$. 

\begin{proposition}\label{prop4.1} $T^{v}(X,Y)$ admits the expression:
\begin{equation}
T^{v}(X,Y)=\frac{-2}{n-1}\Bigl(\langle Y,v\rangle\,X-\langle X,v\rangle\,Y\Bigr).\label{eq4.11}
\end{equation}
Moreover, $T^{v}$ is skew-symmetric (TSS), that is $\dis\langle T^{v}(X,Y),Z\rangle=-\langle T^{v}(Z,Y),X\rangle$
holds for all $X,Y,Z\in\X(M)$ if and only if $v=0$. \end{proposition}

\begin{proof} Using \eqref{eq3.9} and the fact $\nabla_{X}^{0}Y-\nabla_{Y}^{0}X-[X,Y]=0$,
we have 
\[
T^{v}(X,Y)=-\frac{2}{n-1}\bigl(K_{v}(X,Y)-K_{v}(Y,X)\bigr)=\frac{-2}{n-1}\Bigl(\langle Y,v\rangle\,X-\langle X,v\rangle\,Y\Bigr),
\]
that is nothing but \eqref{eq4.11}. Now if for any $X,Y,Z\in\X(M)$,
$\langle T^{v}(X,Y),Z\langle+\langle T^{v}(Z,Y),X\rangle=0$, then
this equality yields 
\[
2\langle Y,v\rangle\langle X,Z\rangle=\langle X,v\rangle\langle Y,Z\rangle+\langle Z,v\rangle\langle Y,X\rangle.
\]
Taking $Y=v$ and $X=Z$ in above equality, we get 
\[
|v|^{2}|X|^{2}=\langle X,v\rangle^{2}.
\]
If $v\neq0$, taking $X$ orthogonal to $v$ yields a contradiction.
\end{proof}

\section{Intrinsic Ricci tensors for Navier-Stokes equations}

\label{sect4}

In what follows, we will denote Levi-Civita covariant derivative by
$\nabla^{0}$. We first compute the Ricci tensor associated to the
connection $\nabla^{v}$.

\begin{proposition}\label{prop4.3} Let $\Ric^{0}$ be the Ricci
curvature associated to $\nabla^{0}$, and $\Ric^{v}$ to $\nabla^{v}$.
Then 
\begin{equation}
\Ric^{v}(X)=\Ric^{0}(X)-\frac{4(n-2)}{(n-1)^{2}}K_{v}(X,v)+\frac{2(n-2)}{n-1}\nabla_{X}^{0}v+\frac{2}{n-1}\div(v)\,X.\label{eq4.12}
\end{equation}
\end{proposition}

\begin{proof} For the sake of simplicity, put $\dis\nabla_{Y}^{v}Z=\nabla_{Y}^{0}Z+S(Y,Z)$,
where $S$ is a $(1,2)$ type tensor on $M$. Then 
\[
\begin{split}\nabla_{X}^{v}\nabla_{Y}^{v}Z & =\nabla_{X}^{0}\nabla_{Y}^{v}Z+S(X,\nabla_{Y}^{v}Z)=\nabla_{X}^{0}\Bigl(\nabla_{Y}^{0}Z+S(Y,Z)\Bigr)+S(X,\nabla_{Y}^{v}Z)\\
 & =\nabla_{X}^{0}\nabla_{Y}^{0}Z+(\nabla_{X}^{0}S)(Y,Z)+S(\nabla_{X}^{0}Y,Z)+S(Y,\nabla_{X}^{0}Z)+S(X,\nabla_{Y}^{v}Z).
\end{split}
\]
Changing role between $X$ and $Y$ yields

\[
\nabla_{Y}^{v}\nabla_{X}^{v}Z=\nabla_{Y}^{0}\nabla_{X}^{0}Z+(\nabla_{Y}^{0}S)(X,Z)+S(\nabla_{Y}^{0}X,Z)+S(X,\nabla_{Y}^{0}Z)+S(Y,\nabla_{X}^{v}Z).
\]
Also 
\[
\nabla_{[X,Y]}^{v}Z=\nabla_{[X,Y]}^{0}Z+S([X,Y],Z).
\]
Combining above equations, the curvature tensor 
\[
R^{v}(X,Y)Z=\nabla_{X}^{v}\nabla_{Y}^{v}Z-\nabla_{Y}^{v}\nabla_{X}^{v}Z-\nabla_{[X,Y]}^{v}Z
\]
which admits the following expression

\[
\begin{split}R^{0}(X,Y)Z & +(\nabla_{X}^{0}S)(Y,Z)-(\nabla_{Y}^{0}S)(X,Z)+S(\nabla_{X}^{0}Y-\nabla_{Y}^{0}X,Z)\\
 & -S(Y,S(X,Z))+S(X,S(Y,Z))-S([X,Y],Z).
\end{split}
\]
Let $x\in M$ and $\{e_{1},\ldots,e_{n}\}$ an orthonormal basis of
$T_{x}M$. Then $\dis\Ric^{v}(X)=\sum_{i=1}^{n}R^{v}(X,e_{i})e_{i}$.
Note that $\dis S(X,Y)=-\frac{2}{n-1}K_{v}(X,Y)$. Put 
\[
I_{1}=\sum_{i=1}^{n}S(X,S(e_{i},e_{i})),\quad I_{2}=\sum_{i=1}^{n}S(e_{i},S(X,e_{i})).
\]

\[
I_{3}=\sum_{i=1}^{n}(\nabla_{X}^{0}S)(e_{i},e_{i}),\quad I_{4}=\sum_{i=1}^{n}(\nabla_{e_{i}}^{0}S)(X,e_{i}).
\]
Then 
\[
\Ric^{v}(X)=\Ric^{0}(X)+I_{1}-I_{2}+I_{3}-I_{4}.
\]

By a completely elementary computation, we find 
\[
\dis I_{1}=\frac{4}{(n-1)^{2}}\sum_{i=1}^{n}K_{v}(X,K_{v}(e_{i},e_{i}))=-\frac{4(n-1)}{(n-1)^{2}}K_{v}(X,v)
\]
and
\[
\dis I_{2}=\frac{4}{(n-1)^{2}}\sum_{i=1}^{n}K_{v}(e_{i},K_{v}(X,e_{i}))=-\frac{4}{(n-1)^{2}}K_{v}(X,v).
\]
For two other terms, 
\[
\dis(\nabla_{X}^{0}S)(Y,Z)=-\frac{2}{n-1}K_{\nabla_{X}^{0}v}(Y,Z)
\]
and
\[
\dis(\nabla_{Y}^{0}S)(X,Z)=-\frac{2}{n-1}K_{\nabla_{Y}^{0}v}(X,Z).
\]
Therefore 
\[
I_{3}=-\frac{2}{n-1}\sum_{i=1}^{n}K_{\nabla_{X}^{0}v}(e_{i},e_{i})=2\nabla_{X}^{0}v.
\]

\[
I_{4}=-\frac{2}{n-1}\sum_{i=1}^{n}K_{\nabla_{e_{i}}^{0}v}(X,e_{i})=-\frac{2}{n-1}\div(v)\,X+\frac{2}{n-1}\nabla_{X}^{0}v.
\]
Finally
\[
\Ric^{v}(X)=\Ric^{0}(X)-\frac{4(n-2)}{(n-1)^{2}}K_{v}(X,v)+\frac{2(n-2)}{n-1}\nabla_{X}^{0}v+\frac{2}{n-1}\div(v)\,X
\]
and the computations are complete.\end{proof}

Since the connection $\nabla^{v}$ has torsion, we have to take account
of torsion tensor into Ricci tensor in a suitable way. A Weitzenböck
formula for a connection which is not of torsion skew-symmetric was
established in \cite{ELL}. Since the dual connection of $\Gamma^{v}$
is not metric, we prefer here avoid to use it. We will define the
so-called \textit{Intrinsic Ricci tensor}, which was firstly introduced
by B. Driver in \cite{Driver}, in the framework of stochastic analysis
on the path space of Riemannian manifolds (see also \cite{Bismut2,FangMalliavin,Hsu,LyonsQian}).
Such a connection was also used in \cite{AM} to obtain an integration
by parts formula for second order differential operators on Riemannian
path spaces.

\begin{definition} The intrinsic Ricci tensor is given by
\begin{equation}
\widehat{\Ric^{v}}(X)=\Ric^{v}(X)+\sum_{i=1}^{n}(\nabla_{e_{i}}^{v}T^{v})(X,e_{i}).\label{eq4.13}
\end{equation}
where $(e_{i})$ is a local orthonormal frame field of the tangent
bundle.\end{definition} 

\begin{theorem}\label{th4.4} Assume that $\dim(M)=3$. Then $\widehat{\Ric^{v}}$
admits the following simple expression: 
\begin{equation}
\widehat{\Ric^{v}}=\Ric^{0}+2v\otimes v+2\nabla^{0,s}v,\label{eq4.14}
\end{equation}
where $\nabla^{0,s}v$ denotes the symmetric part of $\nabla^{0}v$.
\end{theorem}

\begin{proof} By \eqref{eq4.11}, 
\[
\begin{split}(\nabla_{e_{i}}^{v}T)^{v}(X,e_{i}) & =\frac{-2}{n-1}\Bigl(\langle e_{i},\nabla_{e_{i}}^{v}v\rangle\,X-\langle X,\nabla_{e_{i}}^{v}v\rangle\,e_{i}\Bigr)\\
 & =\frac{-2}{n-1}\Bigl(\langle e_{i},\nabla_{e_{i}}^{0}v\rangle\,X-\langle X,\nabla_{e_{i}}^{0}v\rangle\,e_{i}\Bigr)+J_{i},
\end{split}
\]
where
\[
\dis J_{i}=\frac{4}{(n-1)^{2}}\Bigl(\langle e_{i},K_{v}(e_{i},v)\rangle X-\langle X,K_{v}(e_{i},v)\rangle\,e_{i}\Bigr).
\]
Then 
\[
\sum_{i=1}^{n}J_{i}=\frac{4}{(n-1)^{2}}\Bigl((n-1)|v|^{2}X-K_{v}(X,v)\Bigr).
\]
Therefore the sum $\dis\sum_{i=1}^{n}(\nabla_{e_{i}}^{v}T^{v})(X,e_{i})$
is equal to

\[
\frac{-2}{n-1}\Bigl(\div(v)\,X-\sum_{i=1}^{n}\langle X,\nabla_{e_{i}}^{0}v\rangle\,e_{i}\Bigr)+\frac{4}{(n-1)^{2}}\Bigl((n-1)|v|^{2}X-K_{v}(X,)\Bigr).
\]

When $n=3$, the above formula yields that
\begin{equation}
\sum_{i=1}^{3}(\nabla_{e_{i}}^{v}T^{v})(X,e_{i})=-\div(v)\,X+\sum_{i=1}^{3}\langle X,\nabla_{e_{i}}^{0}v\rangle\,e_{i}+2|v|^{2}X-K_{v}(X,v).\label{eq4.15}
\end{equation}

On the other hand, by \eqref{eq4.12}, for $n=3$, 
\begin{equation}
\Ric^{v}(X)=\Ric^{0}(X)-K_{v}(X,v)+\nabla_{X}^{0}v+\div(v)\,X.\label{eq4.16}
\end{equation}
Note that
\[
\dis\sum_{i=1}^{3}\langle X,\nabla_{e_{i}}^{0}v\rangle\,e_{i}+\nabla_{X}^{0}v=\sum_{i=1}^{3}\Bigl(\langle X,\nabla_{e_{i}}^{0}v\rangle+\langle\nabla_{X}^{0}v,e_{i}\rangle\Bigr)\ e_{i}=2\nabla_{X}^{0,s}v.
\]
According to this and summing up \eqref{eq4.15} and \eqref{eq4.16},
we then obtain
\[
\dis\widehat{\Ric^{v}}(X)=\Ric^{0}(X)+2|v|^{2}X-2K_{v}(X,v)+2\nabla_{X}^{0,s}v.
\]
Now remarking that $|v|^{2}X-K_{v}(X,v)=\langle X,v\rangle v$, we
deduce that
\[
\dis\widehat{\Ric^{v}}(X)=\Ric^{0}(X)+2\langle X,v\rangle v+2\nabla_{X}^{0,s}v
\]
for any vector field $X$ and therefore \eqref{eq4.14} holds. \end{proof}

\vskip 4mm \begin{remark} Consider the following SDE on $O(M)$:
\[
dr_{w}(t)=\sqrt{2\nu}\sum_{i=1}^{n}B_{i}(r_{w}(t))\circ dW_{t}^{i},\quad r_{w}(0)=r,
\]
which has its infinitesimal generator 
\[
\nu\sum_{i=1}^{n}\L_{B_{i}}^{2}(f\circ\pi)=\Bigl((\nu\Delta_{M}+2\nu v)f\Bigr)\circ\pi.
\]
According to Equation \eqref{eq3.3}, we have to choose $\dis v=-\frac{1}{2\nu}u_{t}$.
The term $\dis\Ric^{0}-\frac{1}{\nu}\nabla^{0,s}u_{t}$ has already
appeared in resolvent equation \eqref{eq3.4}. In this case, we denote
$\dis\Ric^{t}$ instead of $\Ric^{-u_{t}/2\nu}$ and we have

\begin{equation}
\widehat{\Ric^{t}}=\Ric^{0}+\frac{1}{2\nu^{2}}u_{t}\otimes u_{t}-\frac{1}{\nu}\nabla^{0,s}u_{t}.\label{eq4.17}
\end{equation}
\end{remark}

\begin{proposition}\label{prop4.6} Assume that $\dim(M)=3$. Then 

(i) The following holds:
\begin{equation}
\div(\widehat{\Ric^{t}})=\div(\Ric^{0})+\frac{1}{2\nu^{2}}\nabla_{u_{t}}u_{t}-\frac{1}{\nu}\Ric^{0}u_{t}.\label{eq4.18}
\end{equation}

(ii) Let $\widehat{\Scal^{t}}$ be the associated scalar curvature,
that is $\dis\widehat{\Scal^{t}}=\sum_{i=1}^{n}\langle\widehat{\Ric^{t}}e_{i},\ e_{i}\rangle$
for any orthonormal basis $(e_{i})$ of $T_{x}M$. Then 
\begin{equation}
\widehat{\Scal^{t}}=\Scal^{0}+\frac{1}{2\nu^{2}}|u_{t}|^{2}.\label{eq4.19}
\end{equation}
\end{proposition}

\begin{proof} (i) Since $\div(u_{t})=0$, we have $\div(u_{t}\otimes u_{t})=\nabla_{u_{t}}u_{t}$,
and
\[
\dis\nabla u_{t}=\nabla^{s}u_{t}+\nabla^{sk}u_{t}.
\]
We claim that

\[
\div(\nabla^{sk}u_{t})=-\square u_{t}.
\]
In fact, let $X\in\X(M)$, we have 
\[
\begin{split}\int_{M}\langle\div(\nabla^{sk}u_{t}),X\rangle\,dx & =-\int_{M}\langle\nabla^{sk}u_{t},\nabla X\rangle\ dx=-\int_{M}\langle\nabla^{sk}u_{t},\nabla^{sk}X\rangle\ dx\\
 & =-\int_{M}\langle d\tilde{u}_{t},d\tilde{X}\rangle\,dx=-\int_{M}\langle d^{\ast}d\tilde{u}_{t},\tilde{X}\rangle\,dx=-\int_{M}\langle\square\tilde{u}_{t},\tilde{X}\rangle\,dx.
\end{split}
\]
Therefore
\[
\dis\div(\nabla^{s}u_{t})=\Delta u_{t}+\square u_{t}=\Ric^{0}u_{t}.
\]
The result \eqref{eq4.18} follows.

\vskip 3mm (ii) Concerning \eqref{eq4.19}, by \eqref{eq4.17}, it
is enough to remark that
\[
\dis\sum_{i=1}^{n}\langle\nabla_{e_{i}}^{0,s}u_{t},e_{i}\rangle=\div(u_{t})=0.
\]

\end{proof}

\begin{theorem}\label{prop4.7} Let $\dim(M)=3$, and $(u_{t},\omega_{t})$
be a regular solution to Equation \eqref{eq2.11}. Then the following
identity holds ,

\begin{equation}
\frac{1}{2}\frac{d}{dt}\int_{M}|\omega_{t}|^{2}\,dx+\nu\int_{M}|\nabla^{0}\omega_{t}|^{2}\,dx=\frac{1}{2\nu}\int_{M}(\omega_{t},u_{t})^{2}\,dx-\nu\int_{M}(\widehat{\Ric^{t}}^{\#}\omega_{t},\omega_{t})\ dx.\label{eq4.20}
\end{equation}
where $(\widehat{\Ric^{t}}^{\#}\omega_{t},A)=(\omega_{t},\widehat{\Ric^{t}}A)$
for $A\in\X(M)$. \end{theorem}

\begin{proof} Using \eqref{eq4.17}, 
\[
(\widehat{\Ric^{t}}^{\#}\omega_{t},A)=(\omega_{t},\Ric^{0}A)+\frac{1}{2\nu^{2}}(\omega_{t},u_{t})\langle u_{t},A\rangle-\frac{1}{\nu}(\omega_{t},\nabla_{A}^{0,s}u_{t}).
\]

Note that according to Definition \eqref{eq2.5}, $\dis(\omega_{t},\nabla_{A}^{0,s}u_{t})=(\omega_{t}\intt\nabla^{0,s}u_{t})(A)$.
It follows that 
\begin{equation}
\widehat{\Ric^{t}}^{\#}\omega_{t}=\Ric^{0,\#}\omega_{t}+\frac{1}{2\nu^{2}}(\omega_{t},u_{t})\tilde{u}_{t}-\frac{1}{\nu}\omega_{t}\intt\nabla^{0,s}u_{t}.\label{eq4.21}
\end{equation}

We shall express the right hand side of \eqref{eq2.18} in term of
$\widehat{\Ric^{t}}^{\#}$. By \eqref{eq4.21}, 
\[
\langle\widehat{\Ric^{t}}^{\#}\,\omega_{t},\ \omega_{t}\rangle=\langle\Ric^{0}\,\omega_{t},\omega_{t}\rangle+\frac{1}{2\nu^{2}}(\omega_{t},u_{t})^{2}-\frac{1}{\nu}\langle\omega_{t}\intt\nabla^{0,s}u_{t},\ \omega_{t}\rangle.
\]

Then 
\[
-\nu\,\langle\Ric^{0}\,\omega_{t},\omega_{t}\rangle+\langle\omega_{t}\intt\nabla^{0,s}u_{t},\ \omega_{t}\rangle.=-\nu\,\langle\widehat{\Ric^{t}}^{\#}\,\omega_{t},\ \omega_{t}\rangle+\frac{1}{2\nu}(\omega_{t},u_{t})^{2}.
\]

Substituting this term in the right hand side of \eqref{eq2.18},
we get \eqref{eq4.20}. \end{proof}

\begin{remark} The term $(\omega_{t},u_{t})$ in the right hand side
of \eqref{eq4.20} is called helical density, which involves explicitly
in the evolution of vorticity in time and in space. \end{remark}

\begin{theorem}\label{th2.9} Let $\dim(M)=3$. Then 
\begin{equation}
\begin{split}\frac{d}{dt}\int_{M}(\omega_{t},u_{t})\,dx= & -\nu\int_{M}\langle d\omega_{t},\ *\omega_{t}\rangle_{\Lambda^{2}}\,dx-\nu\int_{M}(\nabla\omega_{t},\nabla u_{t})\,dx\\
 & -\nu\int_{M}(\omega_{t},\widehat{\Ric^{t}}\,u_{t})\,dx+\frac{1}{2\nu}\int_{M}(\omega_{t},u_{t})\,|u_{t}|^{2}\,dx.
\end{split}
\label{eq4.22}
\end{equation}
\end{theorem}

\begin{proof} By \eqref{eq4.17}, 
\[
\ \widehat{\Ric^{t}}\,u_{t}=\Ric^{0}\,u_{t}+\frac{1}{2\nu^{2}}|u_{t}|^{2}\,u_{t}-\frac{1}{\nu}\nabla_{u_{t}}^{0,s}u_{t}.
\]
Hence
\[
\ -\nu\Ric^{0}\,u_{t}+\nabla_{u_{t}}^{0,s}u_{t}=-\nu\,\widehat{\Ric^{t}}\,u_{t}+\frac{1}{2\nu}|u_{t}|^{2}\,u_{t}.
\]
Substituting this term in the right hand of \eqref{eq2.19}, we get
\eqref{eq4.22}. \end{proof}

\section{Case of $\R^{3}$}

\label{sect5}

We will specify results obtained in Section \ref{sect4} on $\R^{n}$.
There are an ocean of publications on Navier-Stokes equations on $\R^{n}$.
We only refer to \cite{Gallagher,Ladyzhenskaya1} for nice expositions
and to \cite{CheminG} for wellposedness of global solutions. We keep
notations used in Section \ref{sect2} for correspondences between
vector fields and differential forms. In this case, $\bigl\{\frac{\partial}{\partial x},\frac{\partial}{\partial y},\frac{\partial}{\partial z}\bigr\}$
form an orthonormal basis at each tangent space of $\R^{3}$, and
$\bigl\{ dx,dy,dz\bigr\}$ an orthonormal basis at each co-tangent
space. Let $u$ be a vector field on $\R^{3}$: $\dis u=u_{1}\frac{\partial}{\partial x}+u_{2}\frac{\partial}{\partial y}+u_{3}\frac{\partial}{\partial z}$,
then $\tilde{u}=u_{1}dx+u_{2}dy+u_{3}dz$ and

\[
\tilde{\omega}=d\tilde{u}=\Bigl(\frac{\partial u_{1}}{\partial z}-\frac{\partial u_{3}}{\partial x}\Bigr)dz\wedge dx+\Bigl(\frac{\partial u_{2}}{\partial x}-\frac{\partial u_{1}}{\partial y}\Bigr)dx\wedge dy+\Bigl(\frac{\partial u_{3}}{\partial y}-\frac{\partial u_{2}}{\partial z}\Bigr)dy\wedge dz.
\]

Hodge star operator gives an isomorphism between $\Lambda^{2}(\R^{3})$
and $\Lambda^{1}(\R^{3})$, we have

\[
\omega=*\tilde{\omega}=\Bigl(\frac{\partial u_{3}}{\partial y}-\frac{\partial u_{2}}{\partial z}\Bigr)dx+\Bigl(\frac{\partial u_{1}}{\partial z}-\frac{\partial u_{3}}{\partial x}\Bigr)dy+\Bigl(\frac{\partial u_{2}}{\partial x}-\frac{\partial u_{1}}{\partial y}\Bigr)dz.
\]
In this case $\omega=\widetilde{{\rm curl\ }}u$, where $\hbox{\rm curl}(u)$
is the curl of $u$, denoted sometimes by $\nabla\times u$. We have
the following relations 
\begin{equation}
\omega=\widetilde{\nabla\times u},\quad\nabla\times(\nabla\times u)=\bigl(d^{\ast}d\tilde{u}\bigr)^{\#}=\bigl(d^{\ast}\tilde{\omega}\bigr)^{\#}.\label{eq5.0}
\end{equation}

By \eqref{eq5.0}, 
\[
\int_{\R^{3}}\langle d\omega_{t},\ *\omega_{t}\rangle_{\Lambda^{2}}\,dx=\int_{\R^{3}}\langle\omega_{t},d^{\ast}(\tilde{\omega})\rangle\,dx=\int_{\R^{3}}\nabla\times(\nabla\times u)\cdot(\nabla\times u)\,dx.
\]

In what follows, we denote $\dis\xi_{t}=\nabla\times u_{t}$. In this
flat case, the intrinsic Ricci tensor $\dis\widehat{\Ric^{t}}$ defined
in Formula \eqref{eq4.17} has expression 
\begin{equation}
\widehat{\Ric^{t}}=\frac{1}{2\nu^{2}}u_{t}\otimes u_{t}-\frac{1}{\nu}\nabla^{s}u_{t},\label{eq5.1}
\end{equation}
where $\nabla^{s}u_{t}$ is the rate of strains. Formula \eqref{eq4.20}
becomes into the following form:

\begin{equation}
\frac{1}{2}\frac{d}{dt}\int_{\R^{3}}|\xi_{t}|^{2}\,dx+\nu\int_{\R^{3}}|\nabla\xi_{t}|^{2}\,dx=\frac{1}{2\nu}\int_{\R^{3}}(\xi_{t}\cdot u_{t})^{2}\,dx-\nu\int_{\R^{3}}(\widehat{\Ric^{t}}\xi_{t},\xi_{t})\ dx.\label{eq5.2}
\end{equation}

This formula says that the variation of vorticity in time and in space
can be explicitly measured by using helicity and the associated intrinsic
Ricci tensor. Formula \eqref{eq4.22} has the form

\begin{equation}
\begin{split}\frac{d}{dt}\int_{\R^{3}}\xi_{t}\cdot u_{t}\,dx= & -\nu\int_{\R^{3}}(\nabla\times\xi_{t})\cdot\xi_{t}\,dx-\nu\int_{\R^{3}}\nabla\xi_{t}\cdot\nabla u_{t}\,dx\\
 & -\nu\int_{\R^{3}}\xi_{t}\cdot\widehat{\Ric^{t}}\,u_{t}\,dx+\frac{1}{2\nu}\int_{\R^{3}}(\xi_{t}\cdot u_{t})\,|u_{t}|^{2}\,dx,
\end{split}
\label{eq5.3}
\end{equation}
which shows how the helicity $\int_{\R^{3}}\xi_{t}\cdot u_{t}\,dx$
varies.

\section{Appendix}

\label{sect6}

\subsection{Proof of Proposition \ref{prop3.1}}

We first give a complete proof of Proposition\ref{prop3.1} by following
the proof of Theorem 3.5.1 in \cite{Hsu}, and emphasize the steps
we have to modify.

\begin{proof} Let $i_{x}$ be the injectivity radius at $x$ and
suppose that

\begin{equation}
i_{M}=\inf\{i_{x};\ x\in M\}>0.\label{eq3.2-7}
\end{equation}
This means that the ball $B(x,i_{M})$ does not meet the cut-locus
$C_{x}$ of $x$. We prepare what we will need for proving \eqref{eq3.2-6}.
\vskip 3mm

Let $x\in B(x_{0},i_{M}/2)^{c}$ which maybe is closed to or in $C_{x_{0}}$.
Let $\gamma_{x}:[0,\eta(x)]\ra M$ be a distance-minimizing geodesic
connecting $x_{0}$ and $x$, parameterized by length. Then $\gamma_{x}(i_{M}/4)\not\in C_{x}$
or $x\not\in C_{\gamma_{x}(i_{M}/4)}$. Put $y=\gamma_{x}(i_{M}/4)$.
Then $d_{M}(x_{0},x)=d_{M}(x_{0},y)+d_{M}(y,x)$. Since $C_{y}$ is
closed, there is $\eps_{0}>0$ such that 
\[
B(x,\eps_{0})\cap C_{y}=\emptyset.
\]
We suppose that such $\eps_{0}$ is valid for all $x$ (in fact, we
will restrict ourselves in a compact set). Let $\eps<\eps_{0}\wedge\frac{i_{M}}{8}$,
and define 
\[
D_{\eps}=\bigl\{ x\in M;\ d_{M}(x,C_{x_{M}})<\eps\bigr\}.
\]
We claim that

\begin{equation}
D_{\eps}\subset B(x_{M},i_{M}/2)^{c}.\label{eq3.2-8}
\end{equation}
In fact, if there exists $x\in D_{\eps}$ such that $d_{M}(x,x_{M})<i_{M}/2$;
there is $z\in C_{x_{M}}$ such that $d_{M}(x,z)<\eps$; then $d_{M}(x_{M},z)\leq d_{M}(x_{M},x)+d_{M}(x,z)<i_{M}$
which contradicts the definition of $i_{M}$. Let $\gamma_{x}$ be
the geodesic considered above. Then $x\not\in C_{y}$ with $y=\gamma_{x}(i_{M}/4)$.

\vskip 3mm Now introduce the stopping times $\sigma_{q}$ by $\sigma_{0}=0$
and 
\[
\sigma_{q}=\inf\bigl\{ t>\sigma_{q-1};\ d_{M}(\pi(r_{t}),\pi(r_{\sigma_{q-1}}))=\eps\bigr\}.
\]
Let $t>0$ and set $t_{q}=t\wedge\sigma_{q}$. Then

\begin{equation}
\rho(r_{t})-\rho(r_{0})=\sum_{q=1}^{+\infty}\Bigl(\rho(r_{t_{q}})-\rho(r_{t_{q-1}})\Bigr).\label{eq3.2-9}
\end{equation}

(i) If $\pi(r_{t_{q-1}})\not\in D_{\eps}$, then for $s\in[t_{q-1},t_{q}]$,
$\pi(r_{s})\not\in C_{x_{M}}$. Applying Itô formula, we have 
\begin{equation}
\rho(r_{t_{q}})-\rho(r_{t_{q-1}})=\sum_{k=1}^{n}\int_{t_{q-1}}^{t_{q}}(\L_{A_{k}}\rho)(r_{s})\,dW_{s}^{k}+\int_{t_{q-1}}^{t_{q}}(L_{s}\rho)(r_{s})\,ds,\label{3.2-10}
\end{equation}
where $\dis L_{s}=\frac{1}{2}\Delta_{O(M)}+\L_{V_{s}}$ \vskip 3mm

(ii) Set $x_{q}=\pi(r_{t_{q}})$. If $x_{q-1}\in D_{\eps}$, then
by discussion at beginning, there is $y_{q-1}$ on a distance-minimizing
geodesic $\gamma$ connecting $x_{M}$ and $x_{q-1}$ such that $\dis d_{M}(x_{M},y_{q-1})=\frac{i_{M}}{4}$
and $x_{q-1}\not\in C_{y_{q-1}}$ and for $s\in[t_{q-1},t_{q}]$,
\[
\dis d_{M}(\pi(r_{s}),x_{q-1})\leq\eps<\eps_{0}.
\]
Therefore $\pi(r_{s})\not\in C_{y_{q-1}}$. Let $\rho_{q}^{\ast}(r)=d_{M}(\pi(r),y_{q-1})$.
Applying Itô formula to $\rho_{q}^{\ast}$, we have

\[
\rho_{q}^{\ast}(r_{t_{q}})-\rho_{q}^{\ast}(r_{t_{q-1}})=\sum_{k=1}^{n}\int_{t_{q-1}}^{t_{q}}(\L_{A_{k}}\rho_{q}^{\ast})(r_{s})\,dW_{s}^{k}+\int_{t_{q-1}}^{t_{q}}(L_{s}\rho_{q}^{\ast})(r_{s})\,ds.
\]
On one hand

\[
d_{M}(x_{M},x_{q-1})=d_{M}(x_{M},y_{q-1})+d_{M}(x_{q-1},y_{q-1})\quad\hbox{or}\quad\rho(r_{t_{q-1}})=\frac{i_{M}}{4}+\rho_{q}^{\ast}(r_{t_{q-1}}),
\]
and on the other hand

\[
d_{M}(x_{M},x_{q})\leq d_{M}(x_{M},y_{q-1})+d_{M}(x_{q},y_{q-1})\quad\hbox{or}\quad\rho(r_{t_{q}})\leq\frac{i_{M}}{4}+\rho_{q}^{\ast}(r_{t_{q}}).
\]
It follows that
\[
\dis\rho(r_{t_{q}})-\rho(r_{t_{q-1}})\leq\rho_{q}^{\ast}(r_{t_{q}})-\rho_{q}^{\ast}(r_{t_{q-1}}).
\]
Therefore there exists $\hat{L}_{q}\geq0$ such that

\[
\rho(r_{t_{q}})-\rho(r_{t_{q-1}})=\rho_{q}^{\ast}(r_{t_{q}})-\rho_{q}^{\ast}(r_{t_{q-1}})-\hat{L}_{q}.
\]

Define 
\[
\tau_{R}=\inf\{t>0,\ d_{M}(x_{M},\pi(r_{t}))>R\}.
\]

As did in \cite{Hsu}, page 95, we get 
\[
\rho(r_{t\wedge\tau_{R})}-\rho(r_{0})=\beta_{t\wedge\tau_{R}}+\int_{0}^{t\wedge\tau_{R}}(L_{s}\rho)(r_{s})\,ds-\hat{L}_{\eps}(t\wedge\tau_{R})+R_{\eps}(t\wedge\tau_{R}),
\]
where 
\[
\hat{L}_{\eps}(t)=\sum_{q=1}^{+\infty}\hat{L}_{q}{\bf 1}_{D_{\eps}}\pi((r_{t_{q-1}}))
\]
which converges to $\hat{L}(t)$ as $\eps\ra0$. The term $R_{\eps}(t)=m_{\eps}(t)+b_{\eps}(t)$
with $m_{\eps}(t)$ the same as in \cite{Hsu}, page 95, so that 
\[
\dis\E(|m_{\eps}(t)|^{2})\leq4\int_{0}^{t}\E({\bf 1}_{D_{2\eps}}(\pi(r_{s})))\,ds.
\]
Therefore for any compact subset $K\subset B(x_{M},R)$,

\[
\begin{split} & \int_{\pi^{-1}(K)}\E(|m_{\eps}(t\wedge\tau_{R})|^{2})\,dr_{0}\leq4\int_{0}^{t}\int_{\pi^{-1}(K)}\E({\bf 1}_{D_{2\eps}}(\pi(r_{s\wedge\tau_{R}})))\,dr_{0}ds\\
 & \ra4\int_{0}^{t}\int_{\pi^{-1}(K)}E({\bf 1}_{C_{x_{M}}}(\pi(r_{s\wedge\tau_{R}})))\,dr_{0}ds\leq4\int_{0}^{t}\!\int_{M}{\bf 1}_{C_{x_{M}}}(x)dxds=0.
\end{split}
\]

The term $b_{\eps}(t)$ has to be modified such that 
\[
b_{\eps}(t)=\sum_{q=1}^{+\infty}\Bigl[\int_{t_{q-1}}^{t_{q}}\Bigl(L_{s}\rho_{q}^{\ast}(r_{s})-L_{s}\rho(r_{s})\Bigr)\,ds\Bigr]{\bf 1}_{D_{\eps}}(\pi(r_{t_{q-1}})).
\]

By \eqref{eq3.2-3} and \eqref{eq3.2-5}, we have to control the term
$1/\rho$. For $x_{q-1}\in D_{\eps}$ and for $s\in[t_{q-1},t_{q}]$,
\[
d_{M}(x_{M},x_{s})\geq d_{M}(x_{M},x_{q-1})-d_{M}(x_{q-1},x_{s})\geq\frac{i_{M}}{2}-\eps\geq\frac{3i_{M}}{8},
\]
and 
\[
d_{M}(y_{q-1},x_{s})\geq d_{M}(x_{M},x_{s})-d_{M}(x_{M},y_{q-1})\geq\frac{3i_{M}}{8}-\frac{i_{M}}{4}=\frac{i_{M}}{8}.
\]
Therefore, according to \eqref{eq3.2-4}, since $x_{s}=\pi(r_{s})\in D_{2\eps}$,
there exists a constant $\alpha>0$ such that 
\[
\int_{t_{q-1}}^{t_{q}}\Bigl|\Bigl(L_{s}\rho_{q}^{\ast}(r_{s})-L_{s}\rho(r_{s})\Bigr)\Bigr|\,ds{\bf 1}_{D_{\eps}}(\pi(r_{t_{q-1}}))\leq\alpha\,\int_{t_{q-1}}^{t_{q}}(1+|V_{s}(r_{s})|){\bf 1}_{D_{2\eps}}(\pi(r_{s}))\,ds.
\]
It follows that

\begin{equation}
\E(|b_{\eps}(t)|)\leq\alpha\,\E\Bigl(\int_{0}^{t}(1+|V_{s}(r_{s})|){\bf 1}_{D_{2\eps}}(\pi(r_{s}))\,ds\Bigr).\label{eq3.2-11}
\end{equation}
Again by hypothesis \eqref{eq2.16}, there is a constant $c_{0}>0$
such that $\hbox{\rm vol}(B(x_{0},\delta))\leq e^{c_{0}\delta}$,
and therefore for a constant $\lambda_{0}>0$, 
\[
C_{M}=\int_{O(M)}\exp(-\lambda_{0}\,d_{M}^{2}(\pi(r_{0}),x_{0}))\,dr_{0}<+\infty.
\]
Define the probability measure $d\mu$ on $O(M)$ by

\begin{equation}
d\mu(r_{0})=\frac{1}{C_{M}}\exp(-\lambda_{0}\,d_{M}^{2}(\pi(r_{0}),x_{0}))\,dr_{0}.\label{measure}
\end{equation}
Now integrating with respect to $d\mu(r_{0})$, we get 
\[
\begin{split} & \int_{0}^{t}\!\int_{\pi^{-1}(K)}\E\Bigl((1+|V_{s}(r_{s})|){\bf 1}_{D_{2\eps}}(\pi(r_{s})){\bf 1}_{(s<\tau_{R})}\Bigr)\,d\mu(r_{0})ds\\
 & \ra\int_{0}^{t}\!\int_{\pi^{-1}(K)}\E\Bigl((1+|V_{s}(r_{s})|){\bf 1}_{C_{x_{M}}}(\pi(r_{s})){\bf 1}_{(s<\tau_{R})}\Bigr)\,d\mu(r_{0})ds\\
 & \leq\sqrt{t}\,\Bigl(\int_{0}^{t}\!\!\int_{M}|v_{s}(x)|^{2}{\bf 1}_{C_{x_{M}}}(x)\,dxds\Bigr)^{1/2}=0,
\end{split}
\]
under the hypothesis \eqref{eq3.2-12}. The proof of Proposition \ref{prop3.1}
is complete. \end{proof}

\subsection{Bismut Formulae and Proof of Proposition \ref{prop2.9}}

In this part, we will first present a nice derivative formulae for
heat semigroup $\TT_{t}$ on differential $p$-forms obtained by Elworthy
and Li in \cite{ElworthyLi} and by Driver and Thalmaier in \cite{DriverTh}.
We keep notations introduced in Section \ref{sect3}. Let $A_{1},\ldots,A_{n}$
be the canonical horizontal vector fields on $O(M)$. Consider the
SDE on $O(M)$

\begin{equation}
dr_{t}=\sum_{i=1}^{n}A_{i}(r_{t})\circ dW_{t}^{i},\quad r_{|_{t=0}}=r_{0}.\label{eq6.1}
\end{equation}
Assume that the Ricci tensor is bounded below $\Ric\geq-\kappa$.
Then SDE \eqref{eq6.1} is stochastic complete (see \cite{Stroock}).
Set $\dis x_{t}=\pi(r_{t})$ with $x_{0}=\pi(r_{0})$. Then $(x_{t})$
is a semi-martingale on $M$, with respect to which stochastic integral
can be defined (see \cite{Bismut1}). Then we can write

\[
dx_{t}=\pi(r_{t})\circ dr_{t}=\sum_{i=1}^{n}d\pi(r_{t})\,A_{i}(r_{t})\circ dW_{t}^{i}=r_{t}\circ dW_{t}.
\]
Therefore $\dis W_{t}=\int_{0}^{t}r_{s}^{-1}\circ dx_{s}$, which
is anti-development of $\{x_{t};t\geq0\}$. Set 
\[
B_{t}=r_{0}\,W_{t}=\int_{0}^{t}//_{s}^{-1}\circ dx_{s},
\]
where $//_{s}=r_{s}\circ r_{0}^{-1}$ is Itô stochastic parallel translation
along path $\{x_{t};t\geq0\}$. Recall that Weitzenböck formula for
$p$-differential forms reads as follows \cite{IW,ElworthyLi}:

\begin{equation}
\square=-\Delta+\RR_{p}^{\#},\label{eq6.2}
\end{equation}
where $\Delta\phi=\hbox{\rm Trace}(\nabla\nabla\phi)$ for a $p$-form
$\phi$, and $\RR_{p}^{\#}:\Lambda^{p}(M)\ra\Lambda^{p}(M)$ is a
tensor, called Weitzenböck curvature. For $p=1$, $\RR_{1}=\Ric^{\#}$
is Ricci tensor. As in \cite{ElworthyLi}, $\RR_{p}(x)$ is an endomorphism
of $p$-vectors, that is, $\dis\RR_{p}(x):\wedge^{p}T_{x}M\ra\wedge^{p}T_{x}M$.
For $r\in O(M)$, define $\dis\hat{\RR}_{p}(r)=r\circ\RR_{p}(\pi(r))\circ r^{-1}$,
more precisely, for $a_{i},b_{j}\in\R^{n}$,

\[
\langle\hat{\RR}_{p}(r)(a_{1}\wedge\cdots\wedge a_{p}),\ b_{1}\wedge\cdots\wedge b_{p}\rangle=\langle\hat{\RR}_{p}(\pi(r))(ra_{1}\wedge\cdots\wedge ra_{p}),\ rb_{1}\wedge\cdots\wedge rb_{p}\rangle.
\]

Consider the heat equation on $p$-forms:

\begin{equation}
\frac{d\phi_{t}}{dt}=-\frac{1}{2}\square\phi_{t},\quad\phi_{|_{t=0}}=\phi_{0}.\label{eq6.3}
\end{equation}

\vskip 3mm

By definition $\TT_{t}\phi_{0}=\phi_{t}$. Consider the following
resolvent equation on $\wedge^{p}\R^{n}$

\begin{equation}
\frac{d\hat{Q}_{t}^{p}}{dt}=-\frac{1}{2}\hat{\RR}_{p}(r_{t})\cdot\hat{Q}_{t}^{p},\quad\hat{Q}_{0}^{p}=\hbox{Id}.\label{eq6.4}
\end{equation}

Define $Q_{t}^{p}:\wedge^{p}(T_{x_{0}}M)\ra\wedge^{p}(T_{x_{t}}M)$
par $Q_{t}^{p}V_{0}=r_{t}\hat{Q}_{t}^{p}(r_{0}^{-1}V_{0})$. It is
well-known (see \cite{ElworthyLi}) that 
\begin{equation}
(\TT_{t}\phi)(V_{0})=\E((\phi_{x_{t}},V_{t}))=\E\Bigl(\langle F_{\phi}(r_{t}),\hat{Q}_{t}^{p}(r_{0}^{-1}V_{0})\rangle\Bigr),\label{eq6.5}
\end{equation}
where $F_{\phi}$ is defined in \eqref{eq3.0} if $\phi$ is a differential
$1$-form, and $F_{\phi}(r)\in\wedge^{p}(\R^{n})$ is such that $\dis\langle F_{\phi}(r),a_{1}\wedge\cdots\wedge a_{p}\rangle=\langle\phi(\pi(r)),ra_{1}\wedge\cdots\wedge ra_{p}\rangle$
where $a_{1},\ldots,a_{p}\in\R^{n}$.

\begin{proposition}\label{prop6.1} Assume that
\begin{equation}
\RR_{p}\geq-\kappa_{p},\quad\kappa\in\R.\label{eq6.5-1}
\end{equation}
Then 
\begin{equation}
|\TT_{t}\phi|\leq e^{\kappa_{p}t/2}|\phi|.\label{eq6.5-2}
\end{equation}
\end{proposition}

\begin{proof} Using \eqref{eq6.4} and \eqref{eq6.5-1}, we have
\[
\frac{d}{dt}\bigl|\hat{Q}_{t}^{p}\,(r_{0}V_{0})\bigr|^{2}=-\langle\hat{\RR}_{p}(r_{t})Q_{t}^{p}\,(r_{0}V_{0}),Q_{t}^{p}\,(r_{0}V_{0})\rangle\leq-\kappa_{p}\bigl|\hat{Q}_{t}^{p}\,(r_{0}V_{0})\bigr|^{2}.
\]

The Gronwall lemma yields that $\dis\bigl|\hat{Q}_{t}^{p}\,(r_{0}V_{0})\bigr|\leq e^{\kappa_{p}t/2}|V_{0}|$.
Since $|F_{\phi}|=|\phi|$, \eqref{eq6.5} yields inequality \eqref{eq6.5-2}.
\end{proof}

\vskip 3mm For simplicity, for $p=1$, we still denote $\kappa$
instead of $\kappa_{1}$. In the case for $1$-forms,

\begin{equation}
|\TT_{t}\phi|\leq e^{\kappa t/2}\TT_{t}^{M}|\phi|.\label{eq6.6}
\end{equation}

To our purpose, we only state the formula for $1$-form established by Elworthy and Li;  although it was stated
 for the case of compact Riemannian manifolds in \cite{ElworthyLi},
but it remains valid in non-compact cases as did by Driver and Thalmaier in \cite{DriverTh}, section 6.

\begin{theorem}\label{th6.2} For $1$-form $\phi$ and a vector
field $v$, 
\begin{equation}
(\square\TT_{t}\phi,v)=-\frac{4}{t^{2}}\E\Bigl[\bigl(\phi_{x_{t}},Q_{t}^{1}\int_{t/2}^{t}(Q_{s}^{1})^{-1}dM_{s}(v)\bigr)\Bigr]\label{eq6.7}
\end{equation}
where $\dis dM_{s}(v)=dM_{s}^{1}(v)+dM_{s}^{2}(v)$ with 
\begin{equation}
dM_{s}^{1}(v)=\theta_{//_{s}dB_{s}}Q_{s}^{2}\Bigl(\int_{0}^{t/2}(Q_{r}^{2})^{-1}\bigl(//_{r}dB_{r}\wedge Q_{r}^{1}(v)\bigr)\Bigr),\label{eq6.8}
\end{equation}
where $\theta$ is annihilation operator, and

\begin{equation}
dM_{s}^{2}(v)=//_{s}dB_{s}\Bigl(\int_{0}^{t/2}\langle Q_{r}^{1}(v),//_{r}dB_{r}\rangle\Bigr).\label{eq6.9}
\end{equation}
\end{theorem}

Let $\{\eps_{1},\ldots,\eps_{n}\}$ be the canonical basis of $\R^{n}$
and set $e_{j}=r_{0}\eps_{j}$. Then $\{e_{1},\ldots,e_{n}\}$ is
an orthonormal basis of $T_{x_{0}}M$. By definition of $\theta$,
the term 
\[
\Bigl<\theta_{//_{s}dB_{s}}\Bigl(Q_{s}^{2}\int_{0}^{t/2}(Q_{r}^{2})^{-1}\bigl(//_{r}dB_{r}\wedge Q_{r}^{1}(v)\bigr)\Bigr),\ //_{s}e_{j}\Bigr>
\]
may be identified with the following
\[
\Big<Q_{s}^{2}\int_{0}^{t/2}(Q_{r}^{2})^{-1}\bigl(//_{r}dB_{r}\wedge Q_{r}^{1}(v)\bigr),\ //_{s}dB_{s}\wedge//_{s}e_{j}\Big>.
\]
Hence
\[
dM_{s}^{1}(v)=\sum_{k,j=1}^{n}\Big<Q_{s}^{2}\int_{0}^{t/2}(Q_{r}^{2})^{-1}\bigl(//_{r}dB_{r}\wedge Q_{r}^{1}(v)\bigr),\ //_{s}e_{k}\wedge//_{s}e_{j}\Big>\ //_{s}e_{j}\,dB_{s}^{k},
\]
and 
\[
dM_{s}^{2}(v)=\sum_{k=1}^{n}\Bigl(\int_{0}^{t/2}\langle Q_{r}^{1}(v),//_{r}dB_{r}\rangle\Bigr)\,//_{s}e_{k}\,dB_{s}^{k}.
\]
Therefore $\dis dM_{s}(v)=\sum_{k=1}^{n}(a_{k}(s)+b_{k}(s))\,dB_{s}^{k}$
with 
\[
a_{k}(s)=\sum_{j=1}^{n}\Big<Q_{s}^{2}\int_{0}^{t/2}(Q_{r}^{2})^{-1}\bigl(//_{r}dB_{r}\wedge Q_{r}^{1}(v)\bigr),\ //_{s}e_{k}\wedge//_{s}e_{j}\Big>\ //_{s}e_{j}.
\]
and $\dis b_{k}(s)=\Bigl(\int_{0}^{t/2}\langle Q_{r}^{1}(v),//_{r}dB_{r}\rangle\Bigr)\,//_{s}e_{k}$.
It is obvious that $\langle a_{k}(s),b_{k}(s)\rangle=0$. 

\begin{lemma}\label{lemma6.3} The quadratic variation $dM_{s}(v)\cdot dM_{s}(v)$
of $M_{s}(v)$ admits the following expression 
\[
dM_{s}(v)\cdot dM_{s}(v)=2\Bigl\| Q_{s}^{2}\int_{0}^{t/2}(Q_{r}^{2})^{-1}\bigl(//_{r}dB_{r}\wedge Q_{r}^{1}(v)\bigr)\Bigr\|_{\Lambda^{2}}^{2}+\Bigl(\int_{0}^{t/2}\langle Q_{r}^{1}(v),//_{r}dB_{r}\rangle\Bigr)^{2}.
\]
\end{lemma}

\begin{theorem}\label{th6.4} Assume that \eqref{eq6.5-1} holds
for $p=1$ and $2$. Then for any differential $1$-form $\phi$,
\begin{equation}
||\TT_{t}\phi||_{2}\leq\frac{2}{t}e^{3\kappa^{+}t/2}\,\sqrt{2(n-1)e^{3\kappa_{2}^{+}t/2}+1}\ ||\phi||_{2},\quad t>0.\label{eq6.12}
\end{equation}
\end{theorem}

\begin{proof} By Theorem \ref{th6.2}, 
\begin{equation}
\begin{split}|(\square\TT_{t}\phi,v)| & \leq\frac{4}{t^{2}}\sqrt{\E(|\phi(x_{t}))|^{2}}\cdot\Bigl(\E\Bigl[\Bigl|Q_{t}^{1}\int_{t/2}^{t}(Q_{s}^{1})^{-1}dM_{s}(v)\Bigr|^{2}\Bigr]\Bigr)^{1/2}\\
 & \leq\frac{4e^{\kappa t/2}}{t^{2}}\sqrt{\E(|\phi(x_{t}))|^{2}}\cdot\Bigl(\E\Bigl[\Bigl|\int_{t/2}^{t}(Q_{s}^{1})^{-1}dM_{s}(v)\Bigr|^{2}\Bigr]\Bigr)^{1/2}.
\end{split}
\label{eq6.13}
\end{equation}
Note that $(Q_{t}^{p})^{-1}$ enjoys the same kind of equations as
\eqref{eq6.4}. Thus $\dis||(Q_{t}^{p})^{-1}||\leq e^{\kappa_{p}t/2}$
under \eqref{eq6.5-1}, so that 
\[
\begin{split}\E\Bigl[\Bigl|\int_{t/2}^{t}(Q_{s}^{1})^{-1}dM_{s}(v)\Bigr|^{2}\Bigr] & \leq\E\Bigl[\int_{t/2}^{t}\sum_{k=1}^{n}\Bigl|(Q_{s}^{1})^{-1}(a_{k}(s)+b_{k}(s))\Bigr|^{2}\Bigr]\\
 & \leq e^{\kappa t}\E\Bigl[\int_{t/2}^{t}dM_{s}(v)\cdot dM_{s}(v)\Bigr]=e^{\kappa t}\bigl(I_{1}(s)+I_{2}(s)\bigr),
\end{split}
\]
where 
\[
I_{1}(s)=\E\Bigl[\int_{t/2}^{t}2\Bigl\| Q_{s}^{2}\int_{0}^{t/2}(Q_{r}^{2})^{-1}\bigl(//_{r}dB_{r}\wedge Q_{r}^{1}(v)\bigr)\Bigr\|_{\Lambda^{2}}^{2}\,ds\Bigr]
\]
\[
I_{2}(s)=\E\Bigl[\int_{t/2}^{t}\Bigl(\int_{0}^{t/2}\langle Q_{r}^{1}(v),//_{r}dB_{r}\rangle\Bigr)^{2}\,ds\Bigr].
\]
It is obvious that $\dis I_{2}(s)\leq\frac{t^{2}e^{\kappa^{+}t/2}}{4}|v|^{2}$
and 
\begin{equation}
I_{1}(s)\leq2e^{\kappa_{2}s}\int_{t/2}^{t}\E\Bigl[\Bigl\|\int_{0}^{t/2}(Q_{r}^{2})^{-1}\bigl(//_{r}dB_{r}\wedge Q_{r}^{1}(v)\bigr)\Bigr\|_{\Lambda^{2}}^{2}\Bigr]\,ds\label{eq6.14}
\end{equation}

Since we have
\[
\dis(Q_{r}^{2})^{-1}\bigl(//_{r}dB_{r}\wedge Q_{r}^{1}(v)\bigr)=\sum_{k=1}^{n}(Q_{r}^{2})^{-1}\bigl(//_{r}e_{k}\wedge Q_{r}^{1}(v)\bigr)dB_{r}^{k},
\]
so that
\[
\begin{split}\E\Bigl[\Bigl\|\int_{0}^{t/2}(Q_{r}^{2})^{-1}\bigl(//_{r}dB_{r}\wedge Q_{r}^{1}(v)\bigr)\Bigr\|_{\Lambda^{2}}^{2}\Bigr] & =\sum_{k=1}^{n}\int_{0}^{t/2}\Bigl\|(Q_{r}^{2})^{-1}\bigl(//_{r}e_{k}\wedge Q_{r}^{1}(v)\bigr)\Bigr\|^{2}\,dr\\
 & \leq\sum_{k=1}^{n}\int_{0}^{t/2}e^{\kappa_{2}r}||//_{r}e_{k}\wedge Q_{r}^{1}(v)||^{2}\,dr.
\end{split}
\]
But
\[
\dis||//_{r}e_{k}\wedge Q_{r}^{1}(v)||^{2}=|Q_{r}^{1}(v)|^{2}-\langle//_{r}e_{k},Q_{r}^{1}(v)\rangle^{2},
\]
we therefore have
\[
\sum_{k=1}^{n}||//_{r}e_{k}\wedge Q_{r}^{1}(v)||^{2}=(n-1)|Q_{r}^{1}(v)|^{2}\leq(n-1)e^{\kappa r}|v|^{2}.
\]
To simplify calculation, we note that $\dis e^{\kappa_{p}r}\leq e^{\kappa_{p}^{+}t/2}$
since $r\in[0,t/2]$. Substituting these bounds first in \eqref{eq6.14},
then together in \eqref{eq6.13}, we finally get

\[
|\square\TT_{t}\phi|\leq\frac{2}{t}e^{3\kappa^{+}t/2}\,\sqrt{2(n-1)e^{3\kappa_{2}^{+}t/2}+1}\,\sqrt{\TT_{t}^{M}|\phi|^{2}},
\]
and the result \eqref{eq6.12} follows. \end{proof}


\begin{thebibliography}{10}
\bibitem{AM} H. Airault, P. Malliavin, Integration by parts formulas
and dilatation vector fields on elliptic probability space, \emph{Prob.
Theory Rel. Fields}, 106 (1996), 447-494.

\bibitem{AC1} M. Arnaudon, A.B. Cruzeiro, Lagrangian Navier-Stokes
diffusions on manifolds: variational principle and stability, \emph{Bull.
Sci. Math.}, 136 (8) (2012), 857--881.

\bibitem{AC2} M. Arnaudon, A.B. Cruzeiro, Stochastic Lagrangian flows
on some compact manifolds, \emph{ Stochastics}, 84 (2012), 367--381.

\bibitem{ACF} M. Arnaudon, A.B. Cruzeiro, S. Fang, Generalized stochastic
Lagrangian paths for the Navier--Stokes equation, \emph{Ann. Sc.
Norm. Super. Pisa}, CI. Sci., 18 (2018), 1033--1060.

\bibitem{Arnold} V. I. Arnold, Sur la géométrie différentielle des
groupes de Lie de dimension infinie et ses applications à l' hydrodynamique
des fluides parfaits, \emph{Ann. Inst. Fourier}, 16 (1966), 316--361.

\bibitem{Bakry} D. Bakry, Etude des transformations de Riesz dans
les variétés à courbure de Ricci minorée, \emph{Sém. Probab.}, XXI,
Lect. notes in Math. 1247 (1987), 137-172.

\bibitem{Bismut1} J.M. Bismut, \emph{Mécanique aléatoire}, {Lect.
Notes in Maths}, 866, Springer-Verlag, 1981.

\bibitem{Bismut2} J.M. Bismut, \emph{Large deviations and Malliavin
calculus}, Birkhäuser, Prog. Math. 45 (1984).

\bibitem{CheminG} J. Y. Chemin, I. Gallagher, On the global wellposedeness
of the 3-D Navier-Stokes equations with large initial data, \emph{Ann.
sci. École Norm. Sup.}, 39 (2006), 679-698.

\bibitem{Cruzeiro} F. Cipriano, A.B. Cruzeiro, { Navier-Stokes equations
and diffusions on the group of homeomorphisms of the torus}, \emph{Comm.
Math. Phys.} 275 (2007), 255--269.

\bibitem{ConstantinI} P. Constantin, G. Iyer, A stochastic Lagrangian
representation of the three-dimensional incompressible Navier-Stokes
equations, \emph{ Comm. Pure Appl. Math.}, 61 (2008), 330--345.

\bibitem{Driver} B. Driver, A Cameron-Martin type quasi-invariance
theorem for Brownian motion on a compact manifold, \emph{J. Funct.
Anal.}, 109 (1992), 272-376.

\bibitem{DriverTh} B. Driver, A. Thalmaier, Heat equation derivative
formulas for vector bundles, \emph{J. Funct. Analysis}, 183 (2001),
42-108.

\bibitem{EM} D.G. Ebin, J.E. Marsden, Groups of diffeomorphisms and
the motion of an incompressible fluid, \emph{Ann. of Math.} 92 (1970),
102--163.

\bibitem{Elworthy} K.D. Elworthy, \emph{Stochastic differential equations
on manifolds}, London Math. Soc. Lect. Note, 70, Cambridge university
Press, 1982.

\bibitem{Elworthy2} J. Eells, K. D. Elworthy, Stochastic dynamical
system, \emph{Control theory and topics in functional analysis}, III,
Intern. atomic energy agency, Vienna, 1976, 179-185.

\bibitem{ELL} K.D. Elworthy, Y. Le Jan, X.M. Li, \emph{On the geometry
of diffusion operators and stochastic flows}, {Lecture Notes in Mathematics},
1720, Springer-Verlag, 1999.

\bibitem{ElworthyLi} K. D. Elworthy, X.M. Li, Bismut formulae for
differential forms, \emph{C. R. Acad. Sci. Paris}, 327 (1998), 87-92.

\bibitem{Fang} S. Fang, Nash embedding, shape operator and Navier-Stokes
equation on a Riemannian manifold, arxive, July 2019.

\bibitem{FangLuo} S. Fang, D. Luo, Constantin and Iyer's representation
formula for the Navier-Stokes equations on manifolds, \emph{Potential
Analysis}, 48 (2018), 181--206.

\bibitem{FLL} S. Fang, H. Li and D. Luo, Heat semi-group and generalized
flows on complete Riemannian manifolds, \emph{Bull. Sci. Math.}, 135
(2011), 565-600.

\bibitem{FangMalliavin} S. Fang, P. Malliavin, Stochastic analysis
on the path space of a Riemannian manifold, \emph{J. Funct. Anal.},
118 (1993), 249-274.

\bibitem{Gallagher} I. Gallagher, Le problème de Cauchy pour les
équations de Navier-Stokes, \emph{Facettes mathématiques de la mécanique
des fluides}, 31-64, Edition Ecole Polytechnique, 2010.

\bibitem{Hsu} E. Hsu, \emph{Stochastic Analysis on Manifolds}, Graduate
Studies in Math., 38 (2002), AMS. \emph{J. Funct. Anal.}, 134 (1995),
417-450.


\bibitem{IW} N. Ikeda, S. Watanabe, \emph{Stochastic differential
equations and diffusion processes}, North-Holland, Math. Library,
24, 1981.

\bibitem{Kobayashi}, M. H. Kobayashi, On the Navier-Stokes equations
on manifolds with curvature, \emph{J Eng Math}, (2008) 60:55--68

\bibitem{Kunita} H. Kunita, \emph{Stochastic flows and Stochastic
differential equations}, Cambridge University Press, 1990.

\bibitem{Ladyzhenskaya1} Ladyzhenskaya, O. A., \emph{The Mathematical
Theory of Viscous Incompressible Flow}, Revised English Edition, Translated
from the Russian by R. A. Silverman, Gordon and Breach (1963).

\bibitem{Li} Xiangdong Li, On the strong $L^p$-Hodge decomposition over complete Riemannian manifolds,
 \emph{J. Funct. Ana}, 257 (2009), 3617-3646.

\bibitem{LyonsQian} T. Lyons, Z. Qian, A class of vector fields on
path spaces, \emph{J. Funct. Anal.}, 145 (1997), 205-223.

\bibitem{Malliavin} P. Malliavin, Formule de la moyenne, calcul des
perturbations et théorie d'annulation pour les formes harmoniques,
\emph{J. Funt. Analysis}, 17 (1974), 274-291.

\bibitem{MT} M. Mitrea, M. Taylor, Navier-Stokes equations on Lipschitz
domains in Riemannian manifolds, \emph{Math. Ann.}, 321(2001), 955--987.

\bibitem{Pier} V. Pierfelice, The incompressible Navier-Stokes equations
on non-compact manifolds, \emph{J. Geom. Anal.}, 27 (2017), 577--617.

\bibitem{Stroock} D. Stroock, \emph{An introduction to the analysis
of paths on a Riemannian manifold}, Mathematical Surveys and Monographs,
74. {American Mathematical Society, Providence, RI}, 2000.

\bibitem{SW} D. Stroock, S.R.S. Varadhan, \emph{Multidimensional
diffusion processes}, Grund. Math. Wissenschaften, 233 (1979), Springer.

\bibitem{Taylor} M. Taylor, \emph{Partial Differential Equations
III: Nonlinear Equations, Nonlinear equations}, Vol. 117, Applied
Mathematical Sciences, Springer New York second edition (2011).

\bibitem{Temam} R. Temam, \emph{Navier-Stokes equations and nonlinear
functional analysis}, Second edition. CBMS-NSF Regional Conference
Series in Applied Mathematics, 66, {Society for Industrial and Applied
Mathematics (SIAM), Philadelphia, PA}, 1995.

\bibitem{TemamW} R. Temam, S. Wang, Inertial forms of Navier-Stokes
equations on the sphere, \emph{J. Funct. Analysis}, 117 (1993), 215--242.

\bibitem{Yano} K. Yano, on harmonic and Killing vector fields, \emph{Ann.
Math.}, 55 (1952), 38-45.

\bibitem{Warner} F. W. Warner, \emph{Foundations of differentiable
manifolds and Lie groups}, Graduate texts in Math., 94 (1983), Springer-Verlag.

\end{thebibliography}
\end{document}